\documentclass[12pt]{article}
\usepackage{amssymb}

\usepackage{amsthm}
\usepackage{amsmath}
\usepackage{graphicx}
\usepackage{enumerate}
\usepackage{amsmath}

\makeatletter
\providecommand*{\cupdot}{%
	\mathbin{%
		\mathpalette\@cupdot{}%
	}%
}
\newcommand*{\@cupdot}[2]{%
	\ooalign{%
		$\m@th#1\cup$\cr
		\hidewidth$\m@th#1\cdot$\hidewidth
	}%
}
\makeatother

\DeclareMathOperator{\DM}{DM}
\DeclareMathOperator{\BDM}{\mathbf{DM}}

\DeclareMathOperator{\Max}{Max}
\DeclareMathOperator{\Min}{Min}

\newcommand{\implik}{\,\,\Rightarrow}

\def\LULU{\hbox{{$L$}{$U$}}}

\newtheorem{theorem}{Theorem}[section]
\newtheorem{definition}[theorem]{Definition}
\newtheorem{lemma}[theorem]{Lemma}
\newtheorem{proposition}[theorem]{Proposition}
\newtheorem{remark}[theorem]{Remark}
\newtheorem{example}[theorem]{Example}
\newtheorem{corollary}[theorem]{Corollary}
\title{Representability of Kleene posets and Kleene lattices}
\author{Ivan~Chajda, Helmut~L\"anger and Jan~Paseka}

\usepackage[hang,footnotesize,bf]{caption}
\usepackage[left, pagewise, edtable]{lineno}

\DeclareCaptionFont{linenumbers}{\ifLineNumbers\internallinenumbers\fi}
\date{}
\begin{document}
%%cislovani stranek	\linenumbers

\footnotetext{Support of the research of all authors by the Austrian Science Fund (FWF), project I~4579-N, and the Czech Science Foundation (GA\v CR), project 20-09869L, entitled ``The many facets of orthomodularity'', is gratefully acknowledged.}

\maketitle

\begin{abstract}
A Kleene lattice is a distributive lattice equipped with an antitone involution and satisfying the so-called normality condition. These lattices were introduced by J.~A.~Kalman. We extended this concept also for posets with an antitone involution. In our recent paper \cite{CLP}, we showed how to construct such Kleene lattices or Kleene posets from a given distributive lattice or poset and a fixed element of this lattice or poset by using the so-called twist product construction, respectively. We extend this construction of Kleene lattices and Kleene posets by considering a fixed subset instead of a fixed element. 

Moreover, we show that in some cases, this generating poset can be embedded into the resulting Kleene poset. 
We investigate the question when a Kleene poset can be represented by a Kleene poset obtained by the 
mentioned construction. We show that a direct product of representable Kleene posets is again representable and hence a direct product of finite chains is representable. This does not hold in general for subdirect products, but we show some examples where it holds. We present large classes of representable and non-representable Kleene posets. Finally, we investigate two kinds of extensions of a distributive poset $\mathbf A$, namely its Dedekind-MacNeille completion $\BDM(\mathbf A)$ and a completion $G(\mathbf A)$ which coincides with $\BDM(\mathbf A)$ provided $\mathbf A$ is finite. In particular we prove that if $\mathbf A$ is a Kleene poset then its extension $G(\mathbf A)$ is  also a Kleene lattice. If the subset $X$ of principal order ideals of $\mathbf A$ is involution-closed and doubly dense in $G(\mathbf A)$ then it generates $G(\mathbf A)$ and it is isomorphic to $\mathbf A$ itself.
\end{abstract}



\section{Introduction}

Kleene lattices serve as algebraic semantics for a specific De Morgan logic. The latter is a logic equipped with a negation $'$ satisfying the double negation law $x''=x$. Here the unary operation $'$ is assumed to be antitone with respect to the induced order, but it need not be a complementation. In order to enrich the properties of such a negation, it is natural to ask the so-called {\em normality condition}, i.e.,\ the inequality
\[
x\wedge x'\leq y\vee y'.
\]
Of course, if $'$ is a complementation, then this inequality is satisfied automatically. But often, the negation in a De Morgan logic has this property, and hence such a negation turns out to be close to complementation. Distributive lattices with an antitone involution satisfying the normality condition are called {\em Kleene lattices} and were introduced by J.~A.~Kalman \cite K (under a different name). To emphasize the importance of this concept, let us note that every MV-algebra, i.e.,\ the algebraic semantics of \L ukasiewicz's many-valued logic, is a Kleene lattice. Moreover, MV-algebras are also crucial in the logic of quantum events because every so-called lattice effect algebra is composed of blocks, which are MV-algebras. Due to this, the question how to construct Kleene lattices  is of some interest and importance.

If instead of lattices, only posets are considered, one obtains so-called {\em Kleene posets}. If we also forget distributivity, we get {\em pseudo-Kleene posets}. We will introduce both notions later.

Our previous paper \cite{CLP} showed how to construct a Kleene lattice $\mathbf K$ from a given distributive lattice $\mathbf L=(L,\vee,\wedge)$ employing the so-called twist product and its reduction using a non-empty subset $S$ of $L$. In such a case, we say that $\mathbf K$ is representable. However, a lot of problems mentioned in \cite{CLP}  remain open. Among them, we would like to try to solve the following ones:
\begin{itemize}
\item Determine classes of representable Kleene lattices as well as classes of not representable Kleene lattices.
\item Can these constructions be extended to Kleene posets?
\item Can every poset be embedded into a Kleene poset obtained by such a construction?
\item Is the Dedekind-MacNeille completion of a representable poset a representable Kleene lattice?
\end{itemize}
The present paper aims to get at least partial answers to the mentioned questions. We show that  direct products of chains can be considered as  representable Kleene lattices and study certain ordinal sums of distributive lattices. We prove that if a pseudo-Kleene poset $\mathbf K$ of odd cardinality can be represented by a distributive poset $\mathbf A$ and a non-empty subset $S$ of $A$, then $S$ must be  a singleton, i.e.,\ $S=\{a\}$, and $\mathbf A$ can be embedded into $\mathbf K$ in such a way that $a$ is mapped onto the unique fixed point of\, {}$'$. We prove further results on representable Kleene posets and Kleene lattices.

First, we recall or introduce several concepts on ordered sets (posets). 

Let $(A,\leq)$ be a poset, $b,c\in A$ and $B,C\subseteq A$. We say
\[
B\leq C\text{ if }x\leq y\text{ for all }x\in B\text{ and }y\in C.
\]
Instead of $B\leq\{c\}$, $\{b\}\leq C$ and $\{b\}\leq\{c\}$ we simply write $B\leq c$, $b\leq C$ and $b\leq c$, respectively. Further, we define
\begin{align*}
L(B) & :=\{x\in A\mid x\leq B\}, \\
U(B) & :=\{x\in A\mid B\leq x\}.
\end{align*}
Instead of $L(B\cup C)$, $L(B\cup\{c\})$, $L(\{b\}\cup C)$, $L(\{b,c\})$ and $L\big(U(B)\big)$ we simply write $L(B,C)$, $L(B,c)$, $L(b,C)$, $L(b,c)$ and $LU(B)$, respectively. Analogously we proceed in similar cases.

A subset $I$ of $A$ is said to be a {\em Frink ideal} if $LU(M) \subseteq  I$
for each finite subset $M \subseteq  I$. Similarly, a 
subset $F$ of $A$ is said to be a {\em Frink filter} if $UL(N) \subseteq  F$
for each finite subset $N \subseteq  F$.

An order-preserving map $f$ between posets $\mathbf A$ and
$\mathbf B$ is said to be
\it an \LULU-morphism \rm\ if
\begin{equation}\label{LUmor}
	\hbox{{$ L$}}\big(f(X)\big)=\hbox{{$ L$}}\big(f(\hbox{{$U$}}\hbox{{$L$}}(X)\big)\quad \hbox{and}\quad%
	\hbox{{$U$}}\big(f(Y)\big)=\hbox{{$U$}}\big(f(\hbox{{$L$}}\hbox{{$ U$}}(Y)\big)
\end{equation}
\noindent
for all non-empty finite subsets $X, Y\subseteq A$.

We say that an \LULU-morphism $f$ is {\em an \LULU-embedding}
({\em \LULU-isomorphism})\
if $f$ is {\em order reflecting} (the inverse map to $f$ is an
\LULU-morphism).

An {\em antitone involution} on $\mathbf A$ is a unary operation $'$ on $A$, satisfying
\begin{align*}
& x\leq y\text{ implies }y'\leq x', \\
& x''=x
\end{align*}
($x,y\in A$).

An element $y \in A$ is said to be a {\em complement} of $x \in A$ 
if $L(x, y) =L(A)$ and $U(x, y) = U(A)$. $\mathbf A$ is said to be 
{\em complemented} if each element of $A$ has a complement in $\mathbf A$. 
$\mathbf A$ is said to be {\em Boolean} if it is distributive and complemented.

If $\mathbf A$ has a greatest element $1$  and a smallest element $0$,  
then an 
antitone involution $'$ on $A$ is called an {\em orthocomplementation} if 
$L(x,x')=\{0\}$ and $U(x,x')=\{1\}$. 

\section{Constructions of  Kleene posets}

Let $\mathbf A=(A,\leq)$ be a poset. Then the {\em twist product} of $\mathbf A$ is defined as $(A^2,\sqsubseteq)$ where
\[
(x,y)\sqsubseteq(z,u)\text{ if }x\leq z\text{ and }u\leq y
\]
($(x,y),(z,u)\in A^2$). It is easy to see that the twist product of $\mathbf A$ is a poset again. This construction was successfully applied in the study of the so-called Nelson-type algebras \cite{BC}.

Now we define the central concept of our paper, which is used to represent Kleene lattices and Kleene posets within twist products. Let $S$ be a non-empty subset of $A$. Define
\[
P_S(\mathbf A):=\{(x,y)\in A^2\mid L(x,y)\leq S\leq U(x,y)\}.
\]
Instead of $P_{\{a\}}(\mathbf A)$, we simply write $P_a(\mathbf A)$.

The ordered triple $(A,\leq,{}')$ is called a {\em pseudo-Kleene poset} if $'$ is an antitone involution on $\mathbf A$ and the {\em normality condition}
\begin{equation}\label{equ1}
L(x,x')\leq U(y,y')\text{ for all }x,y\in A
\end{equation}
holds. The {\em poset} $\mathbf A$ is called {\em distributive} if one of the following equivalent LU-identities is satisfied:
\begin{align*}
     L\big(U(x,y),z\big) & \approx LU\big(L(x,z),L(y,z)\big), \\
U\big(L(x,z),L(y,z)\big) & \approx UL\big(U(x,y),z\big), \\
     U\big(L(x,y),z\big) & \approx UL\big(U(x,z),U(y,z)\big), \\
L\big(U(x,z),U(y,z)\big) & \approx LU\big(L(x,y),z\big), \\
 L\big(U(x_1,x_2, \dots,x_n),z\big) & \approx LU\big(L(x_1,z),L(x_2,z), \dots, L(x_n,z)\big),\\
 U\big(L(x_1,x_2, \dots,x_n),z\big) & \approx UL\big(U(x_1,z),U(x_2,z), \dots, U(x_n,z)\big).
\end{align*}
A {\em Kleene poset} is a distributive pseudo-Kleene poset.

\begin{remark}\rm 
	Recall that 
	D. Zhu in {\cite{Zhu}} introduced the notion of a Kleene poset as an 
	 ordered triple $(A,\leq,{}')$ such that $'$ is an antitone involution on $\mathbf A$ and the 
	 {\em Zhu condition}
	\begin{equation}\label{equzhu}
		x\leq x' \text{ and }  y'\leq y\text{ implies } 
		x\leq y \text{ for all }x,y\in A
	\end{equation}
	holds. In fact, his concept is precisely the pseudo-Kleene poset in our sense because he does not assume distributivity of $(A,\leq)$. 
\end{remark}

\begin{lemma}
	Let $\mathbf A=(A,\leq,{}')$ be a poset with an antitone involution $'$. Then 
	the {normality condition} {\rm (\ref{equ1})} is equivalent to  
	the {Zhu condition}  {\rm (\ref{equzhu})}.
\end{lemma}
\begin{proof}
	Assume first that the {normality condition} holds. Let $x,y\in A, x\leq x',$ and $y'\leq y$. 
	Then $L(x)=L(x,x')\leq U(y,y')=U(y)$. Hence $x\leq y$. 
	
	Now assume that 	the {Zhu condition} holds. Let $x,y\in A$, $u\in L(x,x')$, and 
	$v\in U(y,y')$. Since $u\leq x$ we obtain that $x'\leq u'$. It follows that 
	$u\leq u'$. Similarly, $v'\leq v$. From the {Zhu condition} we conclude $u\leq v$, 
	i.e., $L(x,x')\leq U(y,y')$. 
\end{proof}

Since Kleene algebras are distributive lattices and Zhu does not assume any 
distributivity condition in his definition, we will use the notion of a Kleene poset in our sense. 

For any poset $\mathbf A=(A,\leq)$ and any non-empty subset $S$ of $A$ we define $(x,y)':=(y,x)$ for all $(x,y)\in A^2$ and $\mathbf P_S(\mathbf A):=\big(P_S(\mathbf A),\sqsubseteq,{}')$. Instead of $\mathbf P_{\{a\}}(\mathbf A)$ we simply write $\mathbf P_a(\mathbf A)$.

Let $p_1$ and $p_2$ denote the first and second projection from $P_S(\mathbf A)$ to $A$, respectively.

The following lemma shows how to produce pseudo-Kleene posets from an arbitrarily given poset.

\begin{lemma}
Let $\mathbf A=(A,\leq)$ be a poset and $S$ a non-empty subset of $A$. Then $\mathbf P_S(\mathbf A)$ is a pseudo-Kleene poset and 
\begin{align*}
	L(X) & =\big(L(p_1(X))\times U(p_2(X))\big)\cap  P_S(\mathbf A), \\
	U(X) & =\big(U(p_1(X))\times L(p_2(X))\big)\cap  P_S(\mathbf A)
\end{align*}
for all $X\subseteq P_S(\mathbf A)$.
\end{lemma}

\begin{proof}
Clearly, $\mathbf P_S(\mathbf A)$ is a poset with an antitone involution. Let 
$(a,b),(c,d)\in P_S(\mathbf  A)$, $(e,f)\in L\big((a,b),(b,a)\big)$ and $(g,h)\in U\big((c,d),(d,c)\big)$. Then $e\in L(a,b)$, 
$f\in U(a,b)$, $g\in U(c,d)$, $h\in L(c,d)$, $L(a,b)\leq S\leq U(a,b)$ and $L(c,d)\leq S\leq U(c,d)$ and hence $e\leq S\leq f$ and $h\leq S\leq g$ which implies $e\leq S\leq g$ and $h\leq S\leq f$, i.e.\ $(e,f)\sqsubseteq(g,h)$ showing $L\big((a,b),(b,a)\big)\leq U\big((c,d),(d,c)\big)$.

Let $X\subseteq P_S(\mathbf A)$. Then  the following are equivalent:
\begin{align*}
	&(a,b)\in P_S(\mathbf  A), (a,b) \sqsubseteq X,\\
		&(a,b)\in P_S(\mathbf  A), a \leq p_1(X), b \geq p_2(X),\\
		&(a,b)\in P_S(\mathbf  A), a \in L(p_1(X)), b \in U(p_2(X)),\\
	&(a,b)\in \big(L(p_1(X))\times U(p_2(X))\big)\cap  P_S(\mathbf A).
\end{align*}
Hence 	$L(X)  =\big(L(p_1(X))\times U(p_2(X))\big)\cap  P_S(\mathbf A)$. Similarly, we obtain 
that $U(X) =\big(U(p_1(X))\times L(p_2(X))\big)\cap  P_S(\mathbf A)$.
\end{proof}

Let $\mathbf A=(A,\leq)$ be a poset and $B\subseteq A$. Recall that $\mathbf A$ is said to satisfy the {\em Ascending Chain Condition} {\rm(ACC)} or the {\em Descending Chain Condition} {\rm(DCC)} if in $\mathbf A$, every strictly ascending chain or every strictly descending chain, respectively, is finite. Let $\Max B$ and $\Min B$ denote the set of all maximal and minimal elements of $B$, respectively.

Hence if $\mathbf A=(A,\leq)$ satisfies {\rm(ACC)} or  {\rm(DCC)} then every 
$\emptyset\not = B\subseteq A$ contains maximal or minimal elements, respectively.

\begin{lemma}
Let $\mathbf A=(A,\leq)$ be a poset and $S$ a non-empty subset of $A$.
\begin{enumerate}[{\rm(i)}]
\item If $(S,\leq)$ satisfies the {\rm ACC} and the {\rm DCC}, then $P_S(\mathbf A)=P_{(\Max S)\cup(\Min S)}(\mathbf A)$,
\item if $\bigwedge S$ and $\bigvee S$ exist, then $P_S(\mathbf A)=P_{\{\bigwedge S,\bigvee S\}}(\mathbf A)$.
\end{enumerate}
\end{lemma}

\begin{proof}
\
\begin{enumerate}[(i)]
\item If $(S,\leq)$ satisfies the ACC, then every element of $S$ lies under some maximal element of $S$, and if $(S,\leq)$ meets the DCC, then every element of $S$ lies over some minimal element of $S$.
\item If $\bigwedge S$ and $\bigvee S$ exist then we have
\begin{align*}
P_S(\mathbf A) & =\{(x,y)\in A^2\mid L(x,y)\leq S\leq U(x,y)\}= \\
               & =\{(x,y)\in A^2\mid L(x,y)\leq\bigwedge S\leq\bigvee S\leq U(x,y)\}%
               =P_{\{\bigwedge S,\bigvee S\}}(\mathbf A).
\end{align*}
\end{enumerate}
\end{proof}

A {\em subset} $B$ of a poset $(A,\leq)$ is called {\em convex} if
\[
x,z\in B, y\in A\text{ and }x\leq y\leq z\text{ imply }y\in B.
\]

Let $S\subseteq A$. We put $co(S):=LU(S)\cap UL(S)$.

\begin{lemma}
	Let $\mathbf A=(A,\leq)$ be a poset and $S\subseteq A$. Then $co(S)$ is a convex set 
	containing $S$, 
	$L(S)=L(co(S))$,  $U(S)=U(co(S))$, and $P_S(\mathbf A)=P_{co(S)}(\mathbf A)$.
\end{lemma}
\begin{proof}
Let $x,z\in co(S), y\in A$ and $x\leq y\leq z$. Then $\{x, z\}\leq U(S)$. Since $y\leq z$, we obtain that $\{y\}\leq U(S)$, i.e., $y\in LU(S)$. Similarly, $y\in UL(S)$, i.e., 
$y\in co(S)$. Clearly, $S\subseteq LU(S)\cap UL(S)=co(S)$, $L(co(S))\subseteq L(S)$ 
and $U(co(S))\subseteq U(S)$. We have $co(S)\subseteq UL(S)$ and $co(S)\subseteq LU(S)$. 
We conclude $L(S)=LUL(S)\subseteq L(co(S))$ and $U(S)=ULU(S)\subseteq U(co(S))$, 
i.e., 	$L(S)=L(co(S))$ and  $U(S)=U(co(S))$.

Finally, since $S\subseteq co(S)$ we obtain 
$P_{co(S)}(\mathbf A) \subseteq P_S(\mathbf A)$. Assume now that 
$(x,y)\in P_S(\mathbf A)$. Then $L(x,y)\leq S\leq U(x,y)$. 
Hence $L(x,y)\subseteq L(S)=L(co(S))$ and  $U(x,y)\subseteq U(S)=U(co(S))$. 
We conclude $L(x,y)\leq co(S)\leq U(x,y)$, i.e., $(x,y)\in P_{co(S)}(\mathbf A)$. 
\end{proof}

We are going to show that for a given poset $\mathbf A=(A,\leq)$ and an element $a$ of $A$, the constructed pseudo-Kleene poset $\mathbf P_a(\mathbf A)$ contains $A$ as a convex subset.

\begin{lemma}\label{lem2}
Let $\mathbf A=(A,\leq)$ be a poset and $a\in A$ and let $f$ denote the mapping from $A$ to $P_a(\mathbf A)$ defined by
\[
f(x):=(x,a)\text{ for all }x\in A.
\]
Then $f(A)$ is a convex subset of $\big(P_a(\mathbf A),\sqsubseteq\big)$, 
$\mathbf A$ can be \LULU-embedded into $\big(P_{a}(\mathbf A),\sqsubseteq\big)$, and $f$ is an order isomorphism from $(A,\leq)$ to $\big(f(A),\sqsubseteq\big)$.
\end{lemma}

\begin{proof}
It is clear that $f$ is a mapping from $A$ to $P_a(\mathbf A)$. If $b,c\in A$, $(d,e)\in P_a(\mathbf A)$ and $(b,a)\sqsubseteq(d,e)\sqsubseteq(c,a)$, then $a\leq e\leq a$ and hence $e=a$, which implies $(d,e)\in f(A)$. This shows that $f(A)$ is a convex subset of $\big(f(A),\sqsubseteq\big)$.

Assume now that $X\subseteq A$ is finite and non-empty. 
We compute:
$$\begin{aligned}
U(f(X))&=\big(U(X)\times L(a)\big) \cap P_{a}(\mathbf A)%
=\big(ULU(X)\times L(a)\big) \cap P_{a}(\mathbf A)=U(f(LU(X))),\\
L(f(X))&=\big(L(X)\times U(a)\big) \cap P_{a}(\mathbf A)%
=\big(LUL(X)\times U(a)\big) \cap P_{a}(\mathbf A)=L(f(UL(X))).
\end{aligned}
$$

Finally, for any $x,y\in A$, $x\leq y$ and $(x,a)\sqsubseteq(y,a)$ are equivalent. We conclude that $f$ is an order isomorphism from $(A,\leq)$ to $\big(f(A),\sqsubseteq\big)$.
\end{proof}

\section{Embeddings}

In Lemma~\ref{lem2}, we showed that for a poset $\mathbf A=(A,\leq)$ and an element $a$ of $A$, there is an embedding of $\mathbf A$ into $\big(P_a(\mathbf A),\sqsubseteq\big)$. A similar result can be shown for a distributive lattice $\mathbf L=(L,\vee,\wedge)$ and element $a$ of $L$. However, if the non-empty subset $S$ of $A$ is not a singleton, it is not so easy to find an embedding of $\mathbf L$ into $\big(P_S(\mathbf L),\sqcup,\sqcap\big)$. The following theorem provides a solution to this problem in a particular case.

\begin{theorem}\label{th3}
	Let $\mathbf A=(A,\leq)$ be a distributive poset and $a,b\in A$ with $a\leq b$ and assume that there exists an orthocomplementation 
	${}'$ on $([a,b],\leq)$. Further, assume that for every $x\in A$ satisfying $L(a)\subset L(U(x,a),b)\subset L(b)$ we have $L(U(x,a),b)= L(x)$. Then $\mathbf A$ 
	can be \LULU-embedded into 
	$\big(P_{\{a,b\}}(\mathbf A),\sqsubseteq\big)$ 
	and the  poset $([a,b],\leq, {}')$ is Boolean.
\end{theorem}

%%Recall that another proof of Theorem \ref{th3} is contained in \cite{CLP}. 

\begin{proof}
	Put
	\begin{align*}
		I & :=\{x\in A\mid L(x, b)\leq a\}, \\
		F & :=\{x\in A\mid b\leq U(x, a)\}.
	\end{align*}
	It is easy to see	$a\in I$ and $b\in F$. Let us show that $I$ is a Frink ideal and 
	$F$ is a Frink filter. Assume that $X\subseteq I$, $X$ finite. Let $X=\emptyset$. Then 
	either $LU(X)=\emptyset$ or $LU(X)=\{0\}$ where $0$ is the smallest element of $\mathbf A$. In both 
	cases, $LU(X)\subseteq I$. Suppose now that $X\not=\emptyset$, $X=\{x_1, \dots, x_n\}$ and $x\in LU(X)$. 
	Then $L(x,b)\subseteq L(U(X), b)=LU\big(L(x_1,b),L(x_2,b), \dots, L(x_n,b)\big)\subseteq %
	LU\big(L(a)\big)=L(a)$. Hence $L(x, b)\leq a$. Similarly, $F$ is a Frink filter.

	Let $'$ be an orthocomplementation on $([a,b],\leq)$ and define $f\colon L\rightarrow P_{\{a,b\}}(\mathbf L)$ as follows:
	\[
	f(x):=\left\{
	\begin{array}{ll}
		(x,b)  & \text{if }x\in I, \\
		(x,a)  & \text{if }x\in F, \\
		(x,x') & \text{otherwise}
	\end{array}
	\right.
	\]
	($x\in L$). Of course, $a'=b$ and $b'=a$. 
	
	If $x\in I$, then $L(x, b)\leq a\leq b\leq U(x,b)$ and $(x,b)\in P_{\{a,b\}}(\mathbf A)$. If 
	$x\in F$, then $L(x, a)\leq a\leq b\leq U(x,a)$  and $(x,a)\in P_{\{a,b\}}(\mathbf A)$.
	
	Assume first that $I\cap F\not=\emptyset$. Let $z\in I\cap F$. 
	Then $L(z,b)\leq a$ and $b\leq U(z,a)$. We conclude 
	$$b\in LU(z,a)\cap L(b)=LU\big(L(z,b), L(a,b)\big)\subseteq LU(L(a))=L(a).$$
	Hence $a=b$ and $I=F=A$. Evidently, $f$  is an \LULU-embedding by Lemma~\ref{lem2}.
	
	From now on, we will assume that $I\cap F=\emptyset$. We have that $a<b$.

	Let $c\in A$. If $c\notin I\cup F$ then 
		$L(c, b)\not\subseteq L(a)$ and $U(c, a)\not\subseteq U(b)$, i.e., $ L(b)=LU(b)\not\subseteq LU(c, a)$.

		$LU(c,a)\wedge L(b)=L\big(U(c,a),b\big)=LU\big(L(c,b),L(a,b)\big)%
		=LU\big(L(c,b),a\big)=L(c,b)\vee L(a)$ implies 
	\[
	L(a)\subset L(U(c,a),b)\subset L(b)
	\]
	and hence $L(U(c,a),b)= L(c)$. This shows $a< c < b$. Therefore 
	\[
	A\setminus(I\cup F)\subseteq [a,b]\setminus \{a,b\}.
	\]
	
	If $c\in I\cap[a,b]$, then $a\leq c\in L(c,b)\leq a$, which implies $c=a$. This implies
	\[
	I\cap[a,b]=\{a\}.
	\]
	Dually, we obtain
	\[
	F\cap[a,b]=\{b\}.
	\]
	Therefore, we have a partition 
	$$
	A=I\cupdot F \cupdot ([a,b]\setminus \{a,b\}). 
	$$
	We conclude 
	\[
	f(x)=(x,x')\text{ for all }x\in[a,b].
	\]
	
	Since $'$ is an orthocomplementation on $([a,b],\leq)$ we obtain that 
	$L(x,x')= L(a) \subset L(b) = LU(x,x')$, i.e., 
	$(x,x')\in P_{\{a,b\}}(\mathbf A)$.
	
	We have proved that $f$ is well-deﬁned.

	Let us verify that $f$ is an \LULU-embedding. We first show that 
	$f$ is an \LULU-morphism. Clearly, $f$ is order preserving.

	Let us put $Z_I=Z\cap I$, $Z_{(a,b)}=Z\cap ([a,b]\setminus \{a,b\})$, and 
	$Z_F=Z\cap F$ for any subset $Z\subseteq A$.

	Assume now that $X\subseteq A$ is finite and non-empty. 
 Suppose first that $X_F\not=\emptyset$. Then 
	$LU(X)\cap F\not=\emptyset$ and $U(X)\subseteq F$.

	We compute:
	$$
	U(f(X))=\big(U(X)\times L(a)\big) \cap P_{\{a,b\}}(\mathbf A)%
	=\big(ULU(X)\times L(a)\big) \cap P_{\{a,b\}}(\mathbf A)=U(f(LU(X))).
	$$

	From now on, we will assume that $X_F=\emptyset$.

	Case~1. Let $X_{(a,b)}=\emptyset$. 
	Then $X_I=X$ and $X\subseteq LU(X)\subseteq I$. We compute:
	$$
	\begin{aligned}
	&U\big(f(X)\big)=U(X\times \{b\})=\big(U(X)\times L(b)\big) \cap P_{\{a,b\}}(\mathbf A)\ 
	\text{and}\ U\big(f(LU(X))\big)=\\ %
	&U\big(LU(X)\times \{b\}\big)=\big(ULU(X)\times L(b)\big) \cap P_{\{a,b\}}(\mathbf A)=%
	\big(U(X)\times L(b)\big) \cap P_{\{a,b\}}(\mathbf A).
	\end{aligned}
	$$

	Case~2. Let $X_{(a,b)}\not=\emptyset$.  
	Then $f(X)=X_I\times \{b\}\cup \{(x,x')\mid x\in X_{(a,b)}\}$, 
	$X_{(a,b)}'\subseteq [a,b]\setminus \{a,b\}$
	 and  
	$L(a)\subseteq L(X_{(a,b)}')\subseteq L(b)$. We compute: 
	that 
	$$
	\begin{aligned}
	U\big(f(X)\big)&=\big(U(X)\times L(X_{(a,b)}')\big)\cap P_{\{a,b\}}(\mathbf A)%
	\supseteq 	U(f(LU(X))), \\
	U\big(f(LU(X))\big)&=\begin{cases}\big(ULU(X)\times L(a)\big)\cap P_{\{a,b\}}(\mathbf A),& 
		\text{if } LU(X)\cap F\not=\emptyset,\\
		\big(ULU(X)\times L((LU(X)_{(a,b)})')\big)\cap P_{\{a,b\}}(\mathbf A),&\text{otherwise.}
		\end{cases}\\
	&=\begin{cases}\big(U(X)\times L(a)\big)\cap P_{\{a,b\}}(\mathbf A),& 
%		\phantom{UL}
		\text{if } LU(X)\cap F\not=\emptyset,\\
		\big(U(X)\times L((LU(X)_{(a,b)})')\big)\cap P_{\{a,b\}}(\mathbf A),\phantom{UL}&\text{otherwise.}
	\end{cases}
    \end{aligned}
	$$
	
	Suppose first that $LU(X)\cap F\not=\emptyset$. Assume that $y\in L(X_{(a,b)}')\setminus L(a)$. We have 
	$$
	L(a)\subset L(U(y,a),b)=L(U(y,a))\subseteq L(x')\subset L(b)
	$$
	for every $x\in X_{(a,b)}$. Since $X_{(a,b)}$ is non-empty,  
	we therefore obtain 
	$$
	L(a)\subset L(U(y,a),b)=L(y)\subset L(b),
	$$
	i.e., $a<y <b$ and $a<y' <b$. Moreover, $y\in L(X_{(a,b)}')$ yields $y'\in U(X_{(a,b)})$, i.e., 
	$\emptyset\not=LU(X)\cap F \subseteq LU(X_I,y')\cap F$. Hence there is an element $z\in F$ such that $z\in LU(X_I,y')$ 
	and $b\leq U(z,a)$. Let $X_I=\{x_{i_1}, \dots, x_{i_k}\}$ and 
	$X_{(a,b)}=\{x_{j_1}, \dots, x_{j_l}\}$. 
	We compute:
	$$
		\begin{aligned}
	b\in LU(z,a)\cap L(b)&\subseteq LU\big(LU(X_I, y'),a\big)\cap L(b)=%
	 LU(X_I, y')\cap L(b)\\
	 &=L\big(U(x_{i_1},\dots, x_{i_k}, y'),b\big)=%
	 LU\big(L(x_{i_1},b), \dots, L(x_{i_k},b), L(y',b)\big)\\
	 &\subseteq LU\big(L(a), \dots, L(a), L(y')\big)=L(y'),
	 	\end{aligned}
	$$
	i.e., $b\leq y' < b$, a contradiction. Hence $L(a)= L(X_{(a,b)}')$ and 
	$U\big(f(LU(X))\big)=U\big(f(X)\big)$.
	
	Suppose now that $LU(X)\cap F=\emptyset$.
	
	Let us check that $L(X_{(a,b)}')=L\big(\{u'\mid u\in LU(X)\cap ([a,b]\setminus \{a,b\})\}\big)$.
	
	Clearly, $L(a)\subseteq L\big(\{u'\mid u\in LU(X)\cap ([a,b]\setminus \{a,b\})\}\big) %
	\subseteq L(X_{(a,b)}')\subseteq L(b)$. The first inclusion holds 	
    since $a < u < b$ and $a < u' < b$. The second one follows from the fact that $X_{(a,b)}\subseteq LU(X)$ and $X_{(a,b)}\subseteq ([a,b]\setminus \{a,b\})$. Hence also $X_{(a,b)}\subseteq LU(X)\cap ([a,b]\setminus \{a,b\})\subseteq [a,b]$, i.e.,  $X_{(a,b)}'\subseteq \big(LU(X)\cap ([a,b]\setminus \{a,b\})\big)'$.  
    
    Assume that $y\in L(X_{(a,b)}')\setminus L(a)$. We have 
	$$
	L(a)\subset L\big(U(y,a),b\big)=L\big(U(y,a)\big)\subseteq L(x')\subset L(b)
	$$
	for every $x\in X_{(a,b)}$. Since $X_{(a,b)}$ is non-empty (i.e., 
	at least one $x\in X_{(a,b)}$  exists) we obtain 
	$$
	L(a)\subset L\big(U(y,a),b\big)=L(y)\subset L(b),
	$$
	i.e., $a<y <b$  and $a<y' <b$. Moreover, $y\in L\big(X_{(a,b)}'\big)$ yields $y'\in U\big(X_{(a,b)}\big)$, i.e., $LU\big(X_{(a,b)}\big)\leq y'$.

	Let $u\in LU(X)\cap ([a,b]\setminus \{a,b\})$. We compute:
	$$
		\begin{aligned}
	u\in LU(X)\cap L(b)&=LU\big(L(x_{i_1}, b), \dots, L(x_{i_k}, b), L(x_{j_1}, b), \dots, L(x_{j_l}, b)\big)\\
	&\subseteq LU\big(L(a), \dots, L(a), L(x_{j_1}), \dots, L(x_{j_l}\big))=LU\big(X_{(a,b)}\big)\leq y'.
		\end{aligned}
	$$
	
	 This implies $u\leq y'$, i.e., 
	$y\leq u'$ and $LU(X)\cap L(b)=LU\big(X_{(a,b)}\big)$. Therefore, $y\in L\big(\{u'\mid u\in LU(X)\cap ([a,b]\setminus \{a,b\})\}\big)$. 
	
	We compute:
	$$
	\begin{aligned}
		U\big(f(LU(X))\big)&=U\big(\big(LU(X)\cap I\big)\times \{b\}\cup %
		\{(u,u')\mid u\in LU(X)\cap ([a,b]\setminus \{a,b\})\}\big)\\
		&=\big(ULU(X)\times L\big(\{u'\mid u\in LU(X)\cap ([a,b]\setminus \{a,b\})\}%
		\big)\cap P_{\{a,b\}}(\mathbf A)\\
		&= \big(ULU(X)\times L(X_{(a,b)}')\big)\cap P_{\{a,b\}}(\mathbf A)%
		= \big(U(X)\times L(X_{(a,b)}')\big)\cap P_{\{a,b\}}(\mathbf A)\\
		&=U\big(f(X)\big).
	\end{aligned}
	$$
	
	Similarly, we  obtain that $L\big(f(UL(X))\big)=L\big(f(X)\big)$ for every 
	 finite and non-empty $X\subseteq A$. 
	
%%	Since $f$ is an  \LULU-morphism $f$ preserves order. 

	Assume now that $c, d\in A$, $f(c)=(c,u)\sqsubseteq (d,v)=f(d)$. Then $c\leq d$ and $f$ is order-reflecting.
	
	It remains to check that $([a,b],\leq, {}')$ is a Boolean poset. Clearly, it is 
	complemented. We have to verify that it is distributive. 
	
	Assume that $x, y, z\in [a,b]$. If $\{x, y,z\}\cap \{a,b\}\not=\emptyset$ 
	 then evidently $L_{[a,b]}\big(U_{[a,b]}(x,y),z\big) = %
	  L_{[a,b]}U_{[a,b]}\big(L_{[a,b]}(x,z),L_{[a,b]}(y,z)\big)$. 
	 Hence we may assume that  $a< x, y, z < b$.

	 Let 
	$d\in U_{[a,b]}\big(L_{[a,b]}(x, z), L_{[a,b]}(y, z)\big)$ and 
	$e \in L_{[a,b]}\big(U_{[a,b]}(x, y),z\big)$. 
	We will show that $d\in U\big(L(x, z), L(y, z)\big)$ and 
	$e\in L\big(U(x, y),z\big)$.

	Then 
	$d \geq  L(x, z)$.  Namely, if $d \not\geq  L(x, z)$ then 
	there is $h\in A\setminus [a,b]$ 
	such that $h\not \leq d$, $h< x$ and $h<z$. Hence also $h< b$ and 
	$h\not\leq a$. 

 Clearly, 
	$L(a)\subseteq L\big(U(h,a),b\big)=LU(h,a)\subseteq L(b)$. Assume first $L(a)= LU(h,a)$. Then 
	$h\in LU(h,a)=L(a)$, a contradiction. We have 
	$ LU(h,a)\subseteq L(x)\subset L(b)$. Hence 
	$L(a)\subset L\big(U(h,a),b\big)=L\big(U(h,a)\big)\subset L(b)$, i.e., $a< h <b$, a contradiction again. 
	Therefore really $d \geq  L(x, z)$ and similarly $d \geq  L(y, z)$. Hence 
	$d\in U\big(L(x, z), L(y, z)\big)$.

	Further, $e\in L(z)$ since $e\in  L_{[a,b]}(z)$. Let us check that $e\in LU(x, y)$. Assume that  
	$e\not\in LU(x, y)$. Since 	$e \in L_{[a,b]}\big(U_{[a,b]}(x, y),z\big)$  there exists 
	$g\in U(x,y)\setminus U_{[a,b]}(x, y)$ such that $e\not\leq g$. We have 
	$a < x \leq g\not=b$. Hence 
	$L(a)\subseteq L\big(U(g,a),b\big)=L(g,b)\subseteq L(b)$. 
	Suppose first that 
	$L(g,b)=L(b)$. Then $e\leq b < g$, a contradiction with $e\not\leq g$. Also, 
	$L(a)\subset L(x)\subseteq L(g,b)=L\big(U(g,a),b\big) \subset L(b)$. We conclude $a< g < b$, 
	which is impossible since 	$g\notin  [a,b]$. 
	Therefore $e\in LU(x, y)$ and we obtain 
	$e\in LU(x, y)\cap L(z)=L\big(U(x, y),z\big)$.
	Since $\mathbf A$ is distributive  we see that $e\leq d$. 
	Consequently  $U_{[a,b]}L_{[a,b]}\big(U_{[a,b]}(x,y),z\big)\supseteq 
	U_{[a,b]}\big(L_{[a,b]}(x, z), L_{[a,b]}(y, z)\big)$ 
	which yields that $([a,b],\leq)$ is distributive. 
	\end{proof}

\begin{remark} \rm Note that Theorem  \ref{th3} generalizes 
	\cite[Lemma 6]{CLP} formulated for distributive lattices. We also prefer to use orthocomplementation instead of antitone complementation as in 
	\cite[Lemma 6]{CLP}.
\end{remark}

In the sequel, we use the following notation: If $\mathbf A_1$ and $\mathbf A_2$ are posets with top and bottom elements, respectively, then by $\mathbf A_1+_a\mathbf A_2$ we denote 
the ordinal sum of $\mathbf A_1$ and $\mathbf A_2$ where the top element $a$ of $\mathbf A_1$ is identified with the bottom element of $\mathbf A_2$. If $\mathbf A_1$, $\mathbf A_2$, and $\mathbf A_3$ 
are posets with top element, bottom and top element, and bottom element, respectively, then by 
$\mathbf A_1+_a\mathbf A_2+_b\mathbf A_3$ we denote the ordinal sum of $\mathbf A_1$, $\mathbf A_2$, and $\mathbf A_3$ where the top element $a$ of $\mathbf A_1$ is identified with the bottom element of $\mathbf A_2$ and the top element $b$ of $\mathbf A_2$ is identified with the bottom element of $\mathbf A_3$. For every poset $\mathbf A$, let $\mathbf A^d$ denote its dual.

\begin{lemma}\label{osumdist}
	Let $\mathbf A_1=(A_1,\leq)$ and $\mathbf A_2=(A_2,\leq)$ be distributive posets with top element $a$ and bottom element $a$, respectively. Then $\mathbf A_1+_a\mathbf A_2$ is a distributive poset. 
\end{lemma}
\begin{proof}
Let $x\in A_1$ and $y\in A_2$. Then $L_{\mathbf A_1+_a\mathbf A_2}(x)=L_{\mathbf A_1}(x)$, 
$L_{\mathbf A_1+_a\mathbf A_2}(y)=L_{\mathbf A_2}(y)\cup A_1$, 
$U_{\mathbf A_1+_a\mathbf A_2}(y)=U_{\mathbf A_2}(y)$, 
$U_{\mathbf A_1+_a\mathbf A_2}(x)=U_{\mathbf A_1}(x)\cup A_2$.

	Assume that $x, y, z\in \mathbf A_1+_a\mathbf A_2$. If some pair of elements 
	$x, y, z$ is comparable, then evidently 
	\begin{align*}
		L_{\mathbf A_1+_a\mathbf A_2}\big(U_{\mathbf A_1+_a\mathbf A_2}(x,y),z\big) %
& = L_{\mathbf A_1+_a\mathbf A_2}U_{\mathbf A_1+_a\mathbf A_2}%
\big(L_{\mathbf A_1+_a\mathbf A_2}(x,z),L_{\mathbf A_1+_a\mathbf A_2}(y,z)\big).
		\end{align*}

	Suppose now that there is no pair of comparable elements from 
	$\{x, y, z\}$. Then either 	$\{x, y, z\}\subseteq A_1$ or 
	$\{x, y, z\}\subseteq A_2$. Assume first that $\{x, y, z\}\subseteq A_1$. 
	We compute:
	\begin{align*}
		L_{\mathbf A_1+_a\mathbf A_2}\big(U_{\mathbf A_1+_a\mathbf A_2}(x,y),z\big) %
			& = L_{\mathbf A_1+_a\mathbf A_2}\big(U_{\mathbf A_1}(x,y)\cup A_2,z\big) \\
			& = L_{\mathbf A_1}\big(U_{\mathbf A_1}(x,y),z\big) =%
			L_{\mathbf A_1}U_{\mathbf A_1}\big(L_{\mathbf A_1}(x,z),L_{\mathbf A_1}(y,z)\big)\\
			& =%
			L_{\mathbf A_1}U_{\mathbf A_1}\big(L_{\mathbf A_1+_a\mathbf A_2}(x,z),L_{\mathbf A_1+_a\mathbf A_2}(y,z)\big)\\
		& = L_{\mathbf A_1+_a\mathbf A_2}U_{\mathbf A_1+_a\mathbf A_2}%
		\big(L_{\mathbf A_1+_a\mathbf A_2}(x,z),L_{\mathbf A_1+_a\mathbf A_2}(y,z)\big).
	\end{align*}

The case $\{x, y, z\}\subseteq A_2$ can be verified by the same procedure. 
\end{proof}

The following proposition enables us to determine a broad class of representable Kleene posets by using ordinal sums of distributive posets.

\begin{proposition}\label{prop2}
	Let $\mathbf A_1=(A_1,\leq)$ and $\mathbf A_2=(A_2,\leq)$ be distributive posets with top element $a$ and bottom element $b$, respectively, and $\mathbf B=(B,\leq,{}')$ a  Boolean poset with bottom element $a$ and top element $b$. If $a\not=b$ or $a$ is join-irreducible in $\mathbf A_1$ and meet-irreducible in $\mathbf A_2$ then $\big(P_{ab}(\mathbf A_1+_a\mathbf B+_b\mathbf A_2),\sqsubseteq,{}'\big)$ is 
	a Kleene poset and 
	\[
	\big(P_{ab}(\mathbf A_1+_a\mathbf B+_b\mathbf A_2),\sqsubseteq\big)\cong(\mathbf A_1\times\mathbf A_2^d)+_{(a,b)}\mathbf B+_{(b,a)}(\mathbf A_2\times\mathbf A_1^d).
	\]
\end{proposition}

\begin{proof}
	We have
	\[
	P_{ab}(\mathbf A_1+_a\mathbf B+_b\mathbf A_2)=(A_1\times A_2)\cup\{(x,x')\mid x\in B\}\cup(A_2\times A_1)
	\]
	and the order relations on both sides coincide. The equality follows from the following facts: 
	\begin{enumerate}
		\item if $x\in B\setminus \{a,b\}$ then the only element $y\in A_1\cup B\cup A_2$ satisfying 
		$L_{\mathbf A_1+_a\mathbf B+_b\mathbf A_2}(x,y)\leq a$ and 
		$U_{\mathbf A_1+_a\mathbf B+_b\mathbf A_2}(x,y)\geq b$ is the element $x'\in B$, 
		\item if $x\in A_1$ and $(x,y)\in P_{ab}(\mathbf A_1+_a\mathbf B+_b\mathbf A_2)$ then 
		$y\in A_2$,
		\item if $x\in A_2$ and $(x,y)\in P_{ab}(\mathbf A_1+_a\mathbf B+_b\mathbf A_2)$ then 
		$y\in A_1$, 
		\item $(A_1\times A_2)\cup\{(x,x')\mid x\in B\}\cup(A_2\times A_1)\subseteq %
		P_{ab}(\mathbf A_1+_a\mathbf B+_b\mathbf A_2)$.
	\end{enumerate} 
	
	Since the dual of a distributive poset is again 
	distributive and the cartesian product of distributive posets is distributive, we obtain by Lemma 
	\ref{osumdist} that the ordinal sum $(\mathbf A_1\times\mathbf A_2^d)+_{(a,b)}\mathbf B+_{(b,a)}(\mathbf A_2\times\mathbf A_1^d)$ is also distributive. Hence 
	$\big(P_{ab}(\mathbf A_1+_a\mathbf B+_b\mathbf A_2),\sqsubseteq,{}'\big)$ is 
	a Kleene poset.
\end{proof}

\begin{corollary}\label{cor3}
	Let $\mathbf A_1=(A_1,\leq)$ and $\mathbf A_2=(A_2,\leq)$ be distributive posets with top element $a$ and bottom element $b$, respectively, and $\mathbf B=(B,\leq,{}')$ a  Boolean poset with bottom element $a$ and top element $b$. If $a\not=b$ or $a$ is join-irreducible in $\mathbf A_1$ and meet-irreducible in $\mathbf A_2$ then $\big(P_{ab}(\mathbf A_1+_a\mathbf B),\sqsubseteq,{}'\big)$ and 
	$\big(P_{ab}(\mathbf B+_b\mathbf A_2),\sqsubseteq,{}'\big)$ are  Kleene posets and 
	\begin{align*}
		\big(P_{ab}(\mathbf A_1+_a\mathbf B),\sqsubseteq\big) & \cong\mathbf A_1+_{(a,b)}\mathbf B+_{(b,a)}\mathbf A_1^d, \\
		\big(P_{ab}(\mathbf B+_b\mathbf A_2),\sqsubseteq\big) & \cong\mathbf A_2^d+_{(a,b)}\mathbf B+_{(b,a)}\mathbf A_2.
	\end{align*}
\end{corollary}

\begin{proof}
	It is enough to put $\mathbf A_2:=\mathbf 1$ or $\mathbf A_1:=\mathbf 1$ and use Proposition~\ref{prop2}.
\end{proof}	

\begin{corollary}\label{cor4}
	Let $\mathbf A_1=(A_1,\leq)$ and $\mathbf A_2=(A_2,\leq)$ be distributive posets with top and bottom element $a$, respectively. If $a$ is join-irreducible in $\mathbf A_1$ and meet-irreducible in $\mathbf A_2$ then
	\begin{align*}
		\big(P_a(\mathbf A_1+_a\mathbf A_2),\sqsubseteq\big) & \cong(\mathbf A_1\times\mathbf A_2^d)+_{(a,a)}(\mathbf A_2\times\mathbf A_1^d), \\
		\big(P_a ({\mathbf A}_1),\sqsubseteq\big) & \cong\mathbf A_1+_{(a,a)}\mathbf A_1^d, \\
		\big(P_a(\mathbf A_2),\sqsubseteq\big) & \cong\mathbf A_2^d+_{(a,a)}\mathbf A_2.
	\end{align*}
\end{corollary}

\begin{proof}
	It is enough to put $\mathbf B:=\mathbf 1$ and use Proposition~\ref{prop2} and Corollary~\ref{cor3}. 
\end{proof}

For some of the Kleene posets described in Proposition~\ref{prop2}, we can construct the embeddings as follows:

\begin{corollary}
	Let $\mathbf A_1=(A_1,\leq)$ and $\mathbf A_2=(A_2,\leq)$ be distributive posets with top element $a$ and bottom element $b$, respectively, and $\mathbf B=(B,\leq,{}',a,b)$ a non-trivial bounded Boolean poset, put $\mathbf A:=\mathbf A_1+_a\mathbf B+_b\mathbf A_2$ and define 
	$f\colon A\rightarrow P_{ab}(\mathbf L)$ as follows:
	\[
	f(x):=\left\{
	\begin{array}{ll}
		(x,b)  & \text{if }x\leq a, \\
		(x,x') & \text{if }a\leq x\leq b, \\
		(x,a)  & \text{if }b\leq x
	\end{array}
	\right.
	\]
	{\rm(}$x,y\in A${\rm)}. Then $f$ is an embedding from $\mathbf A$ into $\big(P_{ab}(\mathbf A),\sqsubseteq\big)$, and $f(A)$ is a convex subset of $\big(P_{ab}(\mathbf L),\sqsubseteq\big)$.
\end{corollary}

\begin{proof}
	The first assertion is a special case of Theorem~\ref{th3}. We have
	\begin{align*}
		\mathbf A & =\mathbf A_1+_a\mathbf A+_2\mathbf A_2, \\
		P_{ab}(\mathbf A) & =(A_1\times A_2)\cup\{(x,x')\mid x\in B\}\cup(A_2\times A_1), \\
		f(A) & =(A_1\times\{b\})\cup\{(x,x')\mid x\in B\}\cup(A_2\times\{a\}).
	\end{align*}
	Now assume $(c,d),(h,i)\in f(A)$, $(e,g)\in P_{ab}(\mathbf A)$ and $(c,d)\sqsubseteq(e,g)\sqsubseteq(h,i)$. Then $c\leq e\leq h$ and $a\leq i\leq g\leq d\leq b$, and hence $a\leq g\leq b$. 
	If $e\in B$ then $f(e)=(e,e')=(e,g)$. Assume now that $e< a$. Then $g\geq b$, i.e., $g=b$. 
	We conclude that $f(e)=(e,b)=(e,g)$. Finally, suppose that $b< e$. Hence $g\leq a$, i.e., $g=a$ 
	and $f(e)=(e,a)=(e,g)$.
	% which shows that $(e,g)\in f(A)$. 
	Summing up,  $f(A)$ is a convex subset of $\big(P_{ab}(\mathbf A),\sqsubseteq\big)$.
\end{proof}

\section{Representable Kleene posets}

The following result shows  how to construct representable Kleene posets 
using the direct product of known representable Kleene posets.

\begin{theorem}\label{th4}
	%{\rm(\cite{CLP})} 
	Let $\mathbf A_i=(A_i,\leq)$ be a poset and $S_i$ a non-empty subset of $A_i$ for every $i\in I$. Put
	\[
	\mathbf A:=\prod_{i\in I}\mathbf A_i\text{ and }S:=\prod_{i\in I}S_i.
	\]
	Then 
	\[
	\mathbf P_S(\mathbf A)\cong\prod_{i\in I}\mathbf P_{S_i}(\mathbf A_i).
	\]
	Moreover, if  $\mathbf A_i$ is a distributive poset for every $i\in I$ then  $\mathbf A$ is  a distributive poset.
\end{theorem}

\begin{proof}Recall that, for a non-empty subset $X\subseteq \prod_{i\in I} A_i$, 
	$L_{\mathbf A}(X)=\prod_{i\in I} L_{{\mathbf A}_i}\big(p_i(X)\big)$ and 
	$U_{\mathbf A}(X)=\prod_{i\in I} U_{{\mathbf A}_i}\big(p_i(X)\big)$.

	Let us show that the mapping $f$ from $\prod\limits_{i\in I}P_{S_i}(\mathbf A_i)$ to $P_S(\mathbf A)$ defined by
	\[
	f\big((x_i,y_i)_{i\in I}\big):=\big((x_i)_{i\in I},(y_i)_{i\in I}\big)
	\]
	for all $(x_i,y_i)_{i\in I}\in\prod\limits_{i\in I}P_{S_i}(\mathbf A_i)$ is an isomorphism from $\prod\limits_{i\in I}\mathbf P_{S_i}(\mathbf A_i)$ to $\mathbf P_S(\mathbf A)$. \\
	First, we have to check that $\big((x_i)_{i\in I},(y_i)_{i\in I}\big)\in P_S(\mathbf A)$. \\
	Evidently, $S\not=\emptyset$. Suppose $a=(a_i)_{i\in I}\in S$. Then $a_i\in S_i$ for every $i\in I$. Since $(x_i,y_i)\in P_{S_i}(\mathbf A_i)$ we have that 
	$L_{\mathbf A_i}(x_i, y_i)\leq a_i\leq U_{\mathbf A_i}(x_i, y_i)$. Hence 
	$L_{\mathbf A}\big((x_i)_{i\in I}, (y_i)_{i\in I}\big)\leq a\leq U_{\mathbf A}((x_i)_{i\in I}, (y_i)_{i\in I})$. \\
	Second, let us check that $f$ is an order embedding. Assume $(x_i,y_i)_{i\in I},(u_i,v_i)_{i\in I}\in\prod\limits_{i\in I}P_{S_i}(\mathbf A_i)$. Then the following are equivalent: $(x_i,y_i)_{i\in I}\sqsubseteq(u_i,v_i)_{i\in I}$ in $\prod\limits_{i\in I}\mathbf P_{S_i}(\mathbf A_i)$; $x_i\leq u_i$ and $v_i\leq y_i$ for every $i\in I$; $(x_i)_{i\in I}\leq(u_i)_{i\in I}$ and $(v_i)_{i\in I}\leq(y_i)_{i\in I}$ in $\mathbf A$; $\big((x_i)_{i\in I},(y_i)_{i\in I}\big)\sqsubseteq\big((u_i)_{i\in I},(v_i)_{i\in I}\big)$ in $\mathbf P_S(\mathbf A)$. \\
	Third, we have to verify that $f$ is surjective. Let $z\in P_S(\mathbf A)$. 
	Then $z=(x,y)\in A^2$ and $L_{\mathbf A}(x, y)\leq a\leq U_{\mathbf A}(x, y)$ 
	for all $a\in S$. We have $x=(x_i)_{i\in I}$ and $y=(y_i)_{i\in I}$ with 
	$x_i,y_i\in A_i$ for every $i\in I$. Let us check that $(x_j,y_j)\in P_{S_j}(\mathbf A_j)$ for every $j\in I$. Let $a_j\in S_j$ (this is possible since all $S_i$ are non-empty) and extend this to some $a:=(a_i)_{i\in I}\in S$. Then we have
	\[
	L_{\mathbf A}\big((x_i)_{i\in I}, (y_i)_{i\in I}\big)\leq a\leq U_{\mathbf A}\big((x_i)_{i\in I}, (y_i)_{i\in I}\big).
	\]
	Therefore $L_{\mathbf A_i}(x_i, y_i)\leq a_i\leq U_{\mathbf A_i}(x_i, y_i)$ for all $i\in I$, in particular for $i=j$. \\
	Fourth, let us show that $f$ preserves $'$. Assume that $(x_i,y_i)_{i\in I}\in\prod\limits_{i\in I} P_{S_i}(\mathbf A_i)$. Then $(y_i,x_i)_{i\in I}\in\prod\limits_{i\in I} P_{S_i}(\mathbf A_i)$ and $\big((x_i)_{i\in I},(y_i)_{i\in I}\big),\big((y_i)_{i\in I},(x_i)_{i\in I}\big)\in P_S(\mathbf A)$. We compute
	\begin{align*}
		f\Big(\big((x_i,y_i)_{i\in I}\big)'\Big) & =f\big((y_i,x_i)_{i\in I}\big)=\big((y_i)_{i\in I},(x_i)_{i\in I}\big)=\big((x_i)_{i\in I},(y_i)_{i\in I}\big)'= \\
		& =\Big(f\big((x_i,y_i)_{i\in I}\big)\Big)'.
	\end{align*}
	
	Suppose now that $\mathbf A_i$ is a distributive poset for every $i\in I$.
	Since the distributive law for  
	$\mathbf A$ can be checked componentwise,  $\mathbf A$ is  a distributive poset.
\end{proof}

\begin{lemma}\label{chaininv}
	Let  ${\mathbf C}$	 be a bounded chain with involution and $a\in C$ such that $a\leq a'$ 
	and $x\in [a,a']$ implies $x\in \{a,a'\}$. Then ${\mathbf C}$ is a 
	representable Kleene poset and ${\mathbf C}\cong {\mathbf P}_{[a,a']}([a, 1])$.
\end{lemma}
\begin{proof}
	Assume first that $a=a'$. From \cite[Lemma 18]{CLP} we know that 
	$P_a(\mathbf [a, 1])=\big(\{a\}\times[a,1]\big)\cup\big([a,1]\times\{a\}\big)%
	\cong \mathbf C$. The isomorphism $f$ from $\mathbf C$ to ${\mathbf P}_a(\mathbf [a, 1])$ 
	is given by 
	\begin{align*}
		f(x)=\begin{cases}(a,x')&\text{ if } x\leq a,\\
			(x, a)&\text{ if } a<x.
			\end{cases}
	\end{align*}
Moreover $f(a)=(a,a)$. 

Suppose now that $a<a'$. Let us show that 
${\mathbf C}\cong {\mathbf P}_{\{a,a'\}}([a, 1])$. We define an isomorphism $g$ 
from $\mathbf C$ to ${\mathbf P}_{\{a,a'\}}([a, 1])$ as follows: 
\begin{align*}
	g(x)=\begin{cases}(a,x')&\text{ if } x\leq a,\\
		(x, a)&\text{ if } a<x.
	\end{cases}
\end{align*}
Moreover $g(a)=(a,a')$ and $g(a')=(a',a)$. 
\end{proof}

\begin{remark} \rm Let us denote by ${\mathcal R}{\mathcal C}$ the class of 
	 bounded chains with involution satisfying the assumption of Lemma 
	 \ref{chaininv}. Clearly, all finite chains are in 
	 ${\mathcal R}{\mathcal C}$, which was proved already in 
	 \cite[Corollary 21]{CLP}. Hence due to Theorem \ref{th4} and 
	 Lemma \ref{chaininv} direct products of chains from  ${\mathcal R}{\mathcal C}$ 
	 form a class of representable Kleene lattices.
\end{remark}

By Theorem \ref{th4}, a direct product of representable Kleene 
posets ${\ \mathbf K}_i$ is 
again representable, and the set $S$ for this product is just the 
direct product of the sets $S_i$ for ${\ \mathbf K}_i$. The natural question arises 
if a similar result  also holds for a subdirect product of representable Kleene 
lattices. We can show that, in particular cases, this is true. Let us consider the following example.

\begin{example} Let ${\ \mathbf K}_1$ be the Kleene lattice depicted in Fig.~1 and 
	${\ \mathbf K}_2$ the two-element chain considered as a Kleene lattice: 
	\vspace*{3mm}
	
	\begin{center}
		\setlength{\unitlength}{8mm}
		\begin{tabular}{c c c}
			\begin{picture}(2,4)
				\put(1,0){\circle*{.3}}
				\put(1,1){\circle*{.3}}
				\put(0,2){\circle*{.3}}
				\put(2,2){\circle*{.3}}
				\put(1,3){\circle*{.3}}
				\put(1,4){\circle*{.3}}
				\put(1,1){\line(-1,1)1}
				\put(1,1){\line(0,-1)1}
				\put(1,1){\line(1,1)1}
				\put(1,3){\line(-1,-1)1}
				\put(1,3){\line(1,-1)1}
				\put(1,3){\line(0,1)1}
				\put(1.35,-.3){$0_{{\mathbf K}_1}$}
				\put(1.35,.8){$y$}
				\put(-.6,1.8){$x$}
				\put(-2.399,1.8){${\ \mathbf K}_1=\ $}
				\put(2.35,1.8){$x'$}
				\put(1.35,2.8){$y'$}
				\put(1.35,4.1){$1_{{\mathbf K}_1}$}
				%	\put(.2,-1.4){{\rm Fig.~1}}
			\end{picture}&\phantom{xxxxxccccxxxxxxxxx}&
			\begin{picture}(2,3)
				%	\put(1,0){\circle*{.3}}
				%	\put(0,1){\circle*{.3}}
				%	\put(2,1){\circle*{.3}}
				\put(1,2){\circle*{.3}}
				\put(1,3){\circle*{.3}}
				%	\put(1,0){\line(-1,1)1}
				%	\put(1,0){\line(1,1)1}
				%	\put(1,2){\line(-1,-1)1}
				%	\put(1,2){\line(1,-1)1}
				\put(1,2){\line(0,1)1}
				%	\put(.85,-.6){$a$}
				%	\put(-.6,.8){$c$}
				%	\put(2.35,.8){$d$}
				\put(1.35,1.8){$0_{{\mathbf K}_2}$}
				\put(1.35,3.3){$1_{{\mathbf K}_2}$}
				\put(-1.1399,1.8){${\ \mathbf K}_2=\ $}
				%	\put(.2,-1.4){{\rm Fig.~2}}
			\end{picture}\\
			&&\\
			&{\rm Fig.~1}&
		\end{tabular}
	\end{center}
	
	Then ${\ \mathbf K}_1$ is representable by means of ${\ \mathbf L}_1$ and 
	${S}_1$ and, similarly, ${\ \mathbf K}_2$ is representable by means 
	of ${\ \mathbf L}_2$ and ${S}_2$ as shown in Fig.~2. 
	
	\vspace*{3mm}
	
	\begin{center}
		\setlength{\unitlength}{8mm}
		\begin{tabular}{c c c}
			\begin{picture}(2,4)
				%	\put(1,0){\circle*{.3}}
				\put(1,1){\circle*{.3}}
				\put(0,2){\circle*{.3}}
				\put(2,2){\circle*{.3}}
				\put(1,3){\circle*{.3}}
				\put(1,4){\circle*{.3}}
				\put(1,1){\line(-1,1)1}
				%	\put(1,1){\line(0,-1)1}
				\put(1,1){\line(1,1)1}
				\put(1,3){\line(-1,-1)1}
				\put(1,3){\line(1,-1)1}
				\put(1,3){\line(0,1)1}
				%		\put(1.35,-.3){$0$}
				\put(1.35,.8){$b$}
				\put(-.6,1.8){$a$}
				\put(-2.399,1.8){${\ \mathbf L}_1=\ $}
				\put(2.35,1.8){$a'$}
				\put(1.35,2.8){$b'$}
				\put(1.35,4.1){$1_{{\mathbf L}_1}$}
				%\put(.85,4.3){$1$}
				%	\put(.2,-1.4){{\rm Fig.~1}}
			\end{picture}&\phantom{xxcccxxx}&
			\begin{picture}(2,3)
				%	\put(1,0){\circle*{.3}}
				%	\put(0,1){\circle*{.3}}
				%	\put(2,1){\circle*{.3}}
				\put(1,2){\circle*{.3}}
				\put(1,3){\circle*{.3}}
				%	\put(1,0){\line(-1,1)1}
				%	\put(1,0){\line(1,1)1}
				%	\put(1,2){\line(-1,-1)1}
				%	\put(1,2){\line(1,-1)1}
				\put(1,2){\line(0,1)1}
				%	\put(.85,-.6){$a$}
				%	\put(-.6,.8){$c$}
				%	\put(2.35,.8){$d$}
				\put(1.35,1.8){$0_{{\mathbf L}_2}$}
				\put(1.35,3.3){$1_{{\mathbf L}_2}$}
				%	\put(1.35,1.8){$0$}
				%	\put(.85,3.3){$1$}
				\put(-1.1399,1.8){${\ \mathbf L}_2=\ $}
				%	\put(.2,-1.4){{\rm Fig.~2}}
			\end{picture}\\
			$\begin{array}{r l}
				\multicolumn{2}{c}{S_1=\{a, a', b, b'\}}\\[0.2cm]
				1_{{\mathbf K}_1} \mapsto (1_{\mathbf L_1},b),& y'\mapsto (b',b) \\
				x \mapsto (a,a'),& x'\mapsto (a',a) \\
				0_{{\mathbf K}_1} \mapsto (b,1_{{\mathbf L}_1})%
				,& y\mapsto (b,b') \\
			\end{array}$&&
			$\begin{array}{r l}
				\multicolumn{2}{c}{S_2=\{0_{{\mathbf L}_2}, 1_{{\mathbf L}_2}\}}\\[0.2cm]
				1_{{\mathbf K}_2} \mapsto (1_{{\mathbf L}_2},0_{{\mathbf L}_2}),%
				& 0_{{\mathbf K}_2} \mapsto (0_{{\mathbf L}_2},1_{{\mathbf L}_2})\\
				& \\
				&\\
			\end{array}$\\
			&{\rm Fig.~2}&
		\end{tabular}
	\end{center}
	
	Hence ${{\mathbf K}_1}\cong \mathbf P_{S_1}(\mathbf L_1)$ and 
	${{\mathbf K}_2}\cong \mathbf P_{S_2}(\mathbf L_2)$.

	Consider now the Kleene lattice  ${\mathbf K}={{\mathbf K}_1}\times {{\mathbf K}_2}$. 
	By Theorem \ref{th4} it is representable by means of 
	${\ \mathbf L}={\ \mathbf L}_1\times {{\mathbf L}_2}$ and 
	$S={S}_1\times {{S}_2}$, see Fig.~3. 
	
	\begin{figure}[htp]
		\centering
		\setlength{\unitlength}{9.1mm}
		\begin{tabular}{@{}l c}
			\begin{tabular}{@{}c@{}}				
				\begin{picture}(8,11)
					\put(5,1){\circle*{.3}}
					\put(7,3){\circle*{.3}}
					\put(5,3){\circle*{.3}}
					\put(3,5){\circle*{.3}}
					\put(5,5){\circle*{.3}}
					\put(7,5){\circle*{.3}}
					\put(3,7){\circle*{.3}}
					\put(5,7){\circle*{.3}}
					\put(7,7){\circle*{.3}}
					\put(3,9){\circle*{.3}}
					\put(5,9){\circle*{.3}}
					\put(5,11){\circle*{.3}}
					\put(5,3){\line(-1,1)2}
					\put(5,3){\line(1,1)2}
					\put(5,3){\line(0,1)2}
					\put(3,5){\line(0,1)4}
					\put(3,5){\line(1,1)2}
					\put(5,1){\line(1,1)2}
					\put(5,1){\line(0,1)2}
					\put(5,5){\line(-1,1)2}
					\put(5,5){\line(1,1)2}
					\put(7,3){\line(0,1)2}
					\put(7,5){\line(-1,1)2}
					\put(7,5){\line(0,1)2}
					\put(3,7){\line(1,1)2}
					\put(3,9){\line(1,1)2}
					\put(5,7){\line(0,1)4}
					\put(7,7){\line(-1,1)2}
					\put(5.25,0.85){$(0_{{\mathbf K}_1},0_{{\mathbf K}_2})$}
					\put(5.25,2.85){$(y,0_{{\mathbf K}_2})$}
					\put(0.9,4.85){$(x,0_{{\mathbf K}_2})$}
					\put(5.25,4.85){$(x',0_{{\mathbf K}_2})$}
					\put(7.4,2.85){$(0_{{\mathbf K}_1},1_{{\mathbf K}_2})$}
					\put(7.4,4.85){$(y,1_{{\mathbf K}_2})$}
					\put(0.9,6.85){$(y',0_{{\mathbf K}_2})$}
					\put(5.25,6.85){$(x,1_{{\mathbf K}_2})$}
					\put(7.4,6.85){$(x',1_{{\mathbf K}_2})$}
					\put(5.25,8.85){$(y',1_{{\mathbf K}_2})$}
					\put(0.9,8.85){$(1_{{\mathbf K}_1},0_{{\mathbf K}_2})$}
					%	\put(1.2,10.85){$(d,a)$}
					%	\put(3.225,10.85){$(a,0)$}
					\put(5.4,10.85){$(1_{{\mathbf K}_1},1_{{\mathbf K}_2})$}
					\put(3.4,10.85){${{\mathbf K}=}$}
				\end{picture}\\
				$
				\begin{array}{r@{\,\mapsto\,}l}
					(1_{{\mathbf K}_1},1_{{\mathbf K}_2})&%
					\big((1_{{\mathbf L}_1},1_{{\mathbf L}_2}),%
					(b,0_{{\mathbf L}_2})\big),\\[0.1cm]
					(0_{{\mathbf K}_1},0_{{\mathbf K}_2})&
					\big((b,0_{{\mathbf L}_2}), (1_{{\mathbf L}_1},1_{{\mathbf L}_2})\big),\\[0.1cm] 
					(1_{{\mathbf K}_1},0_{{\mathbf K}_2})&%
					\big((1_{{\mathbf L}_1},0_{{\mathbf L}_2}),%
					(b,1_{{\mathbf L}_2})\big),\\[0.1cm]
					(0_{{\mathbf K}_1},1_{{\mathbf K}_2})&%
					\big((b,1_{{\mathbf L}_2}), %
					(1_{{\mathbf L}_1},0_{{\mathbf L}_2})\big),\\[0.1cm]
					(y',1_{{\mathbf K}_2})&%
					\big((b',1_{{\mathbf L}_2}), 
					(b,0_{{\mathbf L}_2})\big),\\[0.1cm]
					(y,0_{{\mathbf K}_2})&%
					\big((b,0_{{\mathbf L}_2}), 
					(b',1_{{\mathbf L}_2})\big),\\[0.1cm]
					(y',0_{{\mathbf K}_2})&%
					\big((b',0_{{\mathbf L}_2}),%
					(b,1_{{\mathbf L}_2})\big),\\[0.1cm]
					(y,1_{{\mathbf K}_2})&%
					\big((b,1_{{\mathbf L}_2}),%
					(b',0_{{\mathbf L}_2})\big),\\[0.1cm]
					(x,1_{{\mathbf K}_2})&%
					\big((a,1_{{\mathbf L}_2}), 
					(a',0_{{\mathbf L}_2})\big),\\[0.1cm]
					(x',0_{{\mathbf K}_2})&%
					\big((a',0_{{\mathbf L}_2}), 
					(a,1_{{\mathbf L}_2})\big),\\[0.1cm]
					(x',1_{{\mathbf K}_2})&%
					\big((a',1_{{\mathbf L}_2}), 
					(a,0_{{\mathbf L}_2})\big),\\[0.1cm]
					(x,0_{{\mathbf K}_2})&%
					\big((a,0_{{\mathbf L}_2}), 
					(a',1_{{\mathbf L}_2})\big).\\
				\end{array}
				$
			\end{tabular}&%
			\begin{tabular}{@{}c@{}}				
				\begin{picture}(9,11)
					\put(5,3){\circle*{.3}}
					\put(3,5){\circle*{.3}}
					\put(5,5){\circle*{.3}}
					\put(7,5){\circle*{.3}}
					\put(3,7){\circle*{.3}}
					\put(5,7){\circle*{.3}}
					\put(7,7){\circle*{.3}}
					\put(3,9){\circle*{.3}}
					\put(5,9){\circle*{.3}}
					\put(5,11){\circle*{.3}}
					\put(5,3){\line(-1,1)2}
					\put(5,3){\line(1,1)2}
					\put(5,3){\line(0,1)2}
					\put(3,5){\line(0,1)4}
					\put(3,5){\line(1,1)2}
					\put(5,5){\line(-1,1)2}
					\put(5,5){\line(1,1)2}
					\put(7,5){\line(-1,1)2}
					\put(7,5){\line(0,1)2}
					\put(3,7){\line(1,1)2}
					\put(3,9){\line(1,1)2}
					\put(5,7){\line(0,1)4}
					\put(7,7){\line(-1,1)2}
					\put(5.4,2.85){$(b,0_{{\mathbf L}_2})$}
					\put(1.2,4.85){$(a,0_{{\mathbf L}_2})$}
					\put(5.25,4.85){$(a',0_{{\mathbf L}_2})$}
					\put(7.4,4.85){$(b,1_{{\mathbf L}_2})$}
					\put(1.2,6.85){$(b',0_{{\mathbf L}_2})$}
					\put(5.25,6.85){$(a,1_{{\mathbf L}_2})$}
					\put(7.4,6.85){$(a',1_{{\mathbf L}_2})$}
					\put(5.4,8.85){$(b',1_{{\mathbf L}_2})$}
					\put(0.90,8.85){$(1_{{\mathbf L}_1},0_{{\mathbf L}_2})$}
					\put(5.4,10.85){$(1_{{\mathbf L}_1},1_{{\mathbf L}_2})$}
					\put(3.4,10.85){${{\mathbf L}=}$}
				\end{picture}\\
				\begin{picture}(9,6)
					%	\put(3,1){\circle*{.3}}
					%	\put(3,3){\circle*{.3}}
					\put(5,1){\circle*{.3}}
					\put(3,3){\circle*{.3}}
					\put(5,3){\circle*{.3}}
					\put(7,3){\circle*{.3}}
					\put(3,5){\circle*{.3}}
					\put(5,5){\circle*{.3}}
					\put(7,5){\circle*{.3}}
					\put(5,7){\circle*{.3}}
					\put(5,1){\line(-1,1)2}
					\put(5,1){\line(1,1)2}
					\put(5,1){\line(0,1)2}
					\put(3,3){\line(0,1)2}
					\put(3,3){\line(1,1)2}
					\put(5,3){\line(-1,1)2}
					\put(5,3){\line(1,1)2}
					\put(7,3){\line(-1,1)2}
					\put(7,3){\line(0,1)2}
					\put(3,5){\line(1,1)2}
					\put(5,5){\line(0,1)2}
					\put(7,5){\line(-1,1)2}
					\put(5.4,0.85){$(b,0_{{\mathbf L}_2})$}
					\put(1.2,2.85){$(a,0_{{\mathbf L}_2})$}
					\put(5.25,2.85){$(a',0_{{\mathbf L}_2})$}
					\put(7.4,2.85){$(b,1_{{\mathbf L}_2})$}
					\put(1.2,4.85){$(b',0_{{\mathbf L}_2})$}
					\put(5.25,4.85){$(a,1_{{\mathbf L}_2})$}
					\put(7.4,4.85){$(a',1_{{\mathbf L}_2})$}
					\put(5.4,6.85){$(b',1_{{\mathbf L}_2})$}
					\put(3.4,6.85){${{\mathbf S}=}$}
				\end{picture}\\
				$
				\begin{array}{r@{}r@{}l}
					S&{=\{}&(a,0_{{\mathbf L}_2}), (a',0_{{\mathbf L}_2}), %
					(b,0_{{\mathbf L}_2}), (b',0_{{\mathbf L}_2}),\\
					&\phantom{=\{}&(a,1_{{\mathbf L}_2}), (a',1_{{\mathbf L}_2}), %
					(b,1_{{\mathbf L}_2}), (b',1_{{\mathbf L}_2})\}
				\end{array}
				$
			\end{tabular}\\[0.3cm]
			&\\
			\multicolumn{2}{c}{{\rm Fig.~3}}
			%	&{\rm Fig.~13}&%
		\end{tabular}
	\end{figure}
	
	Consider now two subdirect products of ${\mathbf K}$ which are Kleene lattices.
	\begin{enumerate}[{\rm1.)}]
		\item We start with ${{\mathbf K}^s}$, see Fig.~4. 
		
	%	\vspace*{2.95cm}
		
		\begin{figure}[htp]
			\centering
			\setlength{\unitlength}{8.987528mm}
			\begin{tabular}{@{}l@{}c}
				\begin{tabular}{@{}c@{}l@{}}
					\begin{picture}(8,8)
						\put(5,1){\circle*{.3}}
						%	\put(7,3){\circle*{.3}}
						\put(5,3){\circle*{.3}}
						\put(3,5){\circle*{.3}}
						\put(5,5){\circle*{.3}}
						\put(7,5){\circle*{.3}}
						\put(3,7){\circle*{.3}}
						\put(5,7){\circle*{.3}}
						\put(7,7){\circle*{.3}}
						%	\put(3,9){\circle*{.3}}
						\put(5,9){\circle*{.3}}
						\put(5,11){\circle*{.3}}
						\put(5,3){\line(-1,1)2}
						\put(5,3){\line(1,1)2}
						\put(5,3){\line(0,1)2}
						\put(3,5){\line(0,1)2}
						\put(3,5){\line(1,1)2}
						%	\put(5,1){\line(1,1)2}
						\put(5,1){\line(0,1)2}
						\put(5,5){\line(-1,1)2}
						\put(5,5){\line(1,1)2}
						%	\put(7,3){\line(0,1)2}
						\put(7,5){\line(-1,1)2}
						\put(7,5){\line(0,1)2}
						\put(3,7){\line(1,1)2}
						%	\put(3,9){\line(1,1)2}
						\put(5,7){\line(0,1)4}
						\put(7,7){\line(-1,1)2}
						\put(5.25,0.85){$(0_{{\mathbf K}_1},0_{{\mathbf K}_2})$}
						\put(5.25,2.85){$(y,0_{{\mathbf K}_2})$}
						\put(0.9,4.85){$(x,0_{{\mathbf K}_2})$}
						\put(5.25,4.85){$(x',0_{{\mathbf K}_2})$}
						%	\put(7.4,2.85){$(0_{{\mathbf K}_1},1_{{\mathbf K}_2})$}
						\put(7.4,4.85){$(y,1_{{\mathbf K}_2})$}
						\put(0.9,6.85){$(y',0_{{\mathbf K}_2})$}
						\put(5.25,6.85){$(x,1_{{\mathbf K}_2})$}
						\put(7.4,6.85){$(x',1_{{\mathbf K}_2})$}
						\put(5.25,8.85){$(y',1_{{\mathbf K}_2})$}
						%	\put(0.9,8.85){$(1_{{\mathbf K}_1},0_{{\mathbf K}_2})$}
						%	\put(1.2,10.85){$(d,a)$}
						%	\put(3.225,10.85){$(a,0)$}
						\put(5.4,10.85){$(1_{{\mathbf K}_1},1_{{\mathbf K}_2})$}
						\put(3.4,10.85){${{\mathbf K}^s=}$}
						\put(4.35,-.15){Fig.~4}
					\end{picture}&\phantom{xxx}\\
					\multicolumn{2}{c}{\phantom{xxxxxxx}}
				\end{tabular}&
				\begin{tabular}{c l}
					\begin{picture}(8,8)
						\put(5,3){\circle*{.3}}
						\put(3,5){\circle*{.3}}
						\put(5,5){\circle*{.3}}
						\put(7,5){\circle*{.3}}
						\put(3,7){\circle*{.3}}
						\put(5,7){\circle*{.3}}
						\put(7,7){\circle*{.3}}
						%	\put(3,9){\circle*{.3}}
						\put(5,9){\circle*{.3}}
						\put(5,11){\circle*{.3}}
						\put(5,3){\line(-1,1)2}
						\put(5,3){\line(1,1)2}
						\put(5,3){\line(0,1)2}
						\put(3,5){\line(0,1)2}
						\put(3,5){\line(1,1)2}
						\put(5,5){\line(-1,1)2}
						\put(5,5){\line(1,1)2}
						\put(7,5){\line(-1,1)2}
						\put(7,5){\line(0,1)2}
						\put(3,7){\line(1,1)2}
						%	\put(3,9){\line(1,1)2}
						\put(5,7){\line(0,1)4}
						\put(7,7){\line(-1,1)2}
						\put(5.4,2.85){$(b,0_{{\mathbf L}_2})$}
						\put(1.2,4.85){$(a,0_{{\mathbf L}_2})$}
						\put(5.25,4.85){$(a',0_{{\mathbf L}_2})$}
						\put(6.8,4.1985){$(b,1_{{\mathbf L}_2})$}
						\put(1.2,6.85){$(b',0_{{\mathbf L}_2})$}
						\put(5.25,6.85){$(a,1_{{\mathbf L}_2})$}
						\put(6.8,7.485){$(a',1_{{\mathbf L}_2})$}
						\put(5.4,8.85){$(b',1_{{\mathbf L}_2})$}
						%	\put(0.90,8.85){$(1_{{\mathbf L}_1},0_{{\mathbf L}_2})$}
						\put(5.4,10.85){$(1_{{\mathbf L}_1},1_{{\mathbf L}_2})$}
						\put(3.4,10.85){${{\mathbf L}^{s}=}$}
						\put(4.35,-.15){Fig.~5}
					\end{picture}&\phantom{xxx}\\
					\multicolumn{2}{c}{\phantom{xxxxxxx}}
					%\phantom{xx}\\[0.300135cm]
					%\multicolumn{2}{c}{Fig.~5}
				\end{tabular}
			\end{tabular}
		\end{figure}
		
	%	\vspace*{-5mm}
		
		Take $L^{s}:=(L_1\times L_2)\cap K^{s}$ and $S':=S\cap L^{s}$, see 
		Fig.~5. Then 
		$L^{s}=(L_1\times L_2) \setminus \{(1_{{\mathbf L}_1},0_{{\mathbf L}_2})\}$ and $S'=S$.  
		
		It is elementary to show that $\mathbf K^s\cong\mathbf P_{S'}(\mathbf L^s)$.  
		Thus it is representable.
		\item Now, let us consider the subdirect product ${{\mathbf K}^0}$ and put 
		$L^{0}:=(L_1\times L_2)\cap K^{0}$ and $S^{0}:=S\cap L^{0}$ as depicted in Fig.~6.
		
		\vskip0.3cm
		
		\begin{figure}[htp]
			\centering
			\setlength{\unitlength}{9.1mm}
			\begin{tabular}{c c c}
				&&\\
				\begin{picture}(3,4)
					\put(1,0){\circle*{.3}}
					\put(1,1){\circle*{.3}}
					\put(0,2){\circle*{.3}}
					\put(2,2){\circle*{.3}}
					\put(1,3){\circle*{.3}}
					\put(1,4){\circle*{.3}}
					\put(1,1){\line(-1,1)1}
					\put(1,1){\line(0,-1)1}
					\put(1,1){\line(1,1)1}
					\put(1,3){\line(-1,-1)1}
					\put(1,3){\line(1,-1)1}
					\put(1,3){\line(0,1)1}
					\put(1.37,-.3){$(0_{{\mathbf K}_1}, 0_{{\mathbf K}_2})$}
					\put(1.37,.8){$(y, 0_{{\mathbf K}_2})$}
					\put(-2.26,1.8){$(x, 0_{{\mathbf K}_2})$}
					\put(-1.3,3.1){${\ \mathbf K}^{0}=\ $}
					\put(2.37,1.8){$(x', 1_{{\mathbf K}_2})$}
					\put(1.37,2.8){$(y', 1_{{\mathbf K}_2})$}
					\put(1.37,4.1){$(1_{{\mathbf K}_1},1_{\mathbf K_2})$}
					%	\put(.2,-1.4){{\rm Fig.~1}}
				\end{picture}&\begin{tabular}{c c c}
					\phantom{xxxxxxxxxxxxx}
					&\begin{picture}(4,4)
						%	\put(1,0){\circle*{.3}}
						\put(1,1){\circle*{.3}}
						\put(0,2){\circle*{.3}}
						\put(2,2){\circle*{.3}}
						\put(1,3){\circle*{.3}}
						\put(1,4){\circle*{.3}}
						\put(1,1){\line(-1,1)1}
						%	\put(1,1){\line(0,-1)1}
						\put(1,1){\line(1,1)1}
						\put(1,3){\line(-1,-1)1}
						\put(1,3){\line(1,-1)1}
						\put(1,3){\line(0,1)1}
						%		\put(1.35,-.3){$0$}
						\put(1.37,.8){$(b, 0_{{\mathbf L}_2})$}
						\put(-1.9,1.8){$(a, 0_{{\mathbf L}_2})$}
						\put(-0.9099,3.8){${\ \mathbf L}^0=\ $}
						\put(2.35,1.8){$(a', 1_{{\mathbf L}_2})$}
						\put(1.37,2.8){$(b', 1_{{\mathbf L}_2})$}
						\put(1.37,4.1){$(1_{{\mathbf L}_1}, 1_{{\mathbf L}_2})$}
						%\put(.85,4.3){$1$}
						%	\put(.2,-1.4){{\rm Fig.~1}}
					\end{picture}&\phantom{xx}\\
					&\begin{picture}(4,4)
						%	\put(1,0){\circle*{.3}}
						\put(1,1){\circle*{.3}}
						\put(0,2){\circle*{.3}}
						\put(2,2){\circle*{.3}}
						\put(1,3){\circle*{.3}}
						%	\put(1,4){\circle*{.3}}
						\put(1,1){\line(-1,1)1}
						%	\put(1,1){\line(0,-1)1}
						\put(1,1){\line(1,1)1}
						\put(1,3){\line(-1,-1)1}
						\put(1,3){\line(1,-1)1}
						%	\put(1,3){\line(0,1)1}
						%		\put(1.35,-.3){$0$}
						\put(1.37,.8){$(b, 0_{{\mathbf L}_2})$}
						\put(-1.9,1.8){$(a, 0_{{\mathbf L}_2})$}
						\put(-0.9099,2.9){${\ \mathbf S}^0=\ $}
						\put(2.35,1.8){$(a', 1_{{\mathbf L}_2})$}
						\put(1.37,2.8){$(b', 1_{{\mathbf L}_2})$}
						%		\put(1.37,4.1){$(1_{{\mathbf L}_1}, 1_{{\mathbf L}_2})$}
						%\put(.85,4.3){$1$}
						%	\put(.2,-1.4){{\rm Fig.~1}}
					\end{picture}&\end{tabular}&
				\\
				&&\\
				&{\rm Fig.~6}\phantom{xxxxxxxxxxxxxxx}&
			\end{tabular}
		\end{figure}
		Then $\mathbf K^0\cong\mathbf P_{S^0}(\mathbf L^0)$, and 
		${{\mathbf K}^0}$ is representable.
	\end{enumerate}
\end{example}

Note that not every subdirect product of two representable Kleene posets need be representable. On the one hand, if $\mathbf K=(K,\vee,\wedge,{}')$ is a Kleene poset  
that is isomorphic to some $\mathbf P_S(\mathbf A)$, then $\mathbf A$ may not be embeddable into $\mathbf K$ and thus $S$ may not be considered 
as a subset of $K$. On the other hand, every finite distributive lattice is a subdirect product of finite chains and every  finite chain is representable, but  non-representable Kleene lattices exist (see Theorem~\ref{th1}).

\begin{lemma}\label{lem1}
	Let $\mathbf A=(A,\leq)$ be a finite poset and $S$ a non-empty subset of $A$. Then $|P_S(\mathbf A)|$ is odd if and only if $|S|=1$.
\end{lemma}

\begin{proof}
	Let $a,b,c\in A$. Then $(b,c)\in P_S(\mathbf A)$ if and only if $(c,b)\in P_S(\mathbf A)$. If $|S|=1$, say $S=\{a\}$, then $(b,b)\in P_S(\mathbf A)$ if and only if $b=a$. Otherwise, $(b,b)\notin P_S(\mathbf A)$.
\end{proof}

The following three lemmas will be helpful to determine a large class of representable, respectively non-representable, Kleene posets.

\begin{lemma}\label{lem3}
	Let $\mathbf A=(A,\leq)$ be a poset and $a\in A$ and assume $\big(P_a(\mathbf A),\sqsubseteq)$ to have no three-element antichain. Then $a$ is comparable with every element of $A$ and join- and meet-irreducible.
\end{lemma}

\begin{proof} 
	If there would exist some element $b$ of $A$ with $b\parallel a$ then $\{(a,a),(a,b),(b,a)\}$ would be a three-element antichain of $\big(P_a(\mathbf A),\sqsubseteq)$. If $a$ would not be join-irreducible then there would exist $c,d\in A\setminus\{a\}$ with $c\vee d=a$ and $\{(a,a),(c,d),(d,c)\}$ would be a three-element antichain of $\big(P_a(\mathbf A),\sqsubseteq)$. If, finally, $a$ would not be meet-irreducible then there would exist $e,f\in A\setminus\{a\}$ with $e\wedge f=a$ and $\{(a,a),(e,f),(f,e)\}$ would be a three-element antichain of $\big(P_a(\mathbf A),\sqsubseteq)$.
\end{proof}

\begin{lemma}\label{lem5}
	Let $\mathbf K=(K,\leq,{}')$ be a finite pseudo-Kleene poset, $\mathbf A=(A,\leq)$ a poset and $S$ a non-empty subset of $A$. Then
	\begin{enumerate}[{\rm(i)}]
		\item the antitone involution $'$ on $K$ has at most one fixed point,
		\item the antitone involution $'$ on $K$ has a fixed point if and only if $|K|$ is odd,
		\item the antitone involution $'$ on $P_S(\mathbf A)$ has a fixed point if and only if $|S|=1$.
	\end{enumerate}
\end{lemma}

\begin{proof}
	Let $a,b\in K$ and $c,d,e\in A$.
	\begin{enumerate}[(i)]
		\item If $a'=a$ and $b'=b$ then $L(a) = L(a,a')\leq U(b,b')=U(b)$ 
		and $ L(b)= L(b,b')\leq U(a,a')= U(a)$ and hence $a=b$.
		\item The set
		\[
		\bigcup_{x\in K}\{x,x'\}^2
		\]
		is an equivalence relation on $K$ having only two-element classes if $'$ on $K$ has no fixed point and having only two-element classes, but precisely 
		one one-element class otherwise.
		\item We have $(c,c)\in P_c(\mathbf A)$ and $(c,c)'=(c,c)$ in $P_c(\mathbf A)$. Now assume $|S|>1$. If $(d,e)$ would be a fixed point of $'$ in $P_S(\mathbf A)$ then $d=e$ and we would have $L(d)=L(d,d)\leq S\leq U(d,d)= U(d)$ and hence $d\leq s\leq d$, i.e., $d=s$ for all $s\in S$ contradicting $|S|>1$. Hence $'$ has no fixed point in $P_S(\mathbf A)$.
	\end{enumerate} 
\end{proof}

 In the following lemma, we derive an upper bound for $|A|$ provided $P_a(\mathbf A)$ is finite. This result will be used in the following theorem describing representable Kleene posets of odd cardinality.

\begin{lemma}\label{lem4}
	Let $\mathbf A=(A,\leq)$ be a poset and $a\in A$ and assume 
	$P_a(\mathbf A)$ to be finite. Then $A$ is finite and $|A|<|P_a(\mathbf A)|/2+1$.
\end{lemma}

\begin{proof}
	Since $P_a(\mathbf A)\supseteq(\{a\}\times A)\cup(A\times\{a\})$ we have that $A$ 
	is finite and $|P_a(\mathbf A)|\geq2|A|-1$ whence $|A|\leq(|P_a(\mathbf A)|+1)/2<|P_a(\mathbf A)|/2+1$.
\end{proof}

\begin{theorem}\label{th2}
	Let $\mathbf K=(K,\leq,{}')$ be a finite representable Kleene poset with an odd number of elements. Then $'$ has exactly one fixed point $a$, and there exists some subposet $\mathbf A$ of $(K,\leq)$ of cardinality less than $|K|/2+1$ containing $a$ such that $\mathbf P_a(\mathbf A)\cong\mathbf K$.
\end{theorem}

\begin{proof}
	Since $\mathbf K$ is representable, there exists some poset $\mathbf A^*=(A^*,\leq)$ and some non-empty subset $S$ of $A^*$ such that $\mathbf P_S(\mathbf A^*)\cong\mathbf K$. According to Lemma~\ref{lem5}, $'$ has a fixed point in $K$ and hence $|S|=1$, say $S=\{b\}$, again because of Lemma~\ref{lem5}. According to Lemma~\ref{lem4}, $A^*$ is finite, and $|A^*|<|K|/2+1$. Let $f$ denote an isomorphism from $P_b(\mathbf A^*)$ to $\mathbf K$. Obviously, $(b,b)$ is the unique fixed point of $'$ in $P_b(\mathbf A^*)$, and hence $f(b,b)=a$. Let $g$ denote the embedding $x\mapsto(x,b)$ of $\mathbf A^*$ into $(P_b(\mathbf A^*),\sqsubseteq)$. Then $f\circ g$ is an embedding of $\mathbf A^*$ into $(K,\leq)$ mapping $b$ onto $a$. Hence, if $A:=f\big(g(A^*)\big)$ then $\mathbf A:=(A,\leq)$ is a subposet of $(K,\leq)$ isomorphic to $\mathbf A^*$ and $\mathbf P_a(\mathbf A)\cong \mathbf P_b(\mathbf A^*)\cong\mathbf K$.
\end{proof}

The following theorem shows a class of non-representable Kleene lattices.

\begin{theorem}\label{th1}
	Let $\mathbf C$ be a finite chain containing more than one element and $\mathbf B$ the four-element Boolean algebra and let $\mathbf K=(K,\vee,\wedge,{}')$ denote the Kleene lattice $\mathbf C+_b\mathbf B+_c\mathbf B+_d\mathbf C$. Then $\mathbf K$ is not representable.
\end{theorem}

\begin{proof}
	Using the method of indirect proof, let us suppose $\mathbf K$ to be representable. According to Theorem~\ref{th2}, there exists some subposet $\mathbf A=(A,\leq)$ of $(K,\vee,\wedge)$ containing the element $e\in A$ such that $\mathbf P_e(\mathbf A)\cong\mathbf K$ and 
	$(e,e)\mapsto c$. Moreover, $e$ is both  meet-irreducible and  join-irreducible. 
	Corollary~\ref{cor4} shows that $\mathbf A$ cannot be a chain. Hence $A$ must contain 
	two non-comparable elements $u$ and $v$ such that $u$ and $v$ are 
	comparable with $e$. Hence either 
	$u,v\in U(e)$ or $u,v\in L(e)$. 
	In the first case $(u,e)$ and  $(v,e)$ cover $(e,e)$ and $(u,e)\sqcap (v,e)=(e,e)$, i.e., $u\wedge v=e$  in 
	$\mathbf A$, a contradiction with meet-irreducibility of $e$. 
	In the second case $(u,e)$ and  $(v,e)$ are covered by $(e,e)$ and $(u,e)\sqcup (v,e)=(e,e)$, i.e., $u\vee v=e$  in 
	$\mathbf A$, a contradiction with join-irreducibility of $e$. This shows that $\mathbf K$ is not representable.
\end{proof}

Note that since the proof of Theorem~\ref{th1} uses Lemma~\ref{lem3} and since in Lemma~\ref{lem3} we have the assumption that $\big(P_a(\mathbf A),\sqsubseteq)$ does not have a three-element antichain, one cannot replace the four-element Boolean algebra $\mathbf B$ in Theorem~\ref{th1} by a larger Boolean algebra.

\section{Dedekind-MacNeille completion of Kleene posets}

Recall that the {\em Dedekind-MacNeille completion} $\BDM(\mathbf A)$ of a poset $\mathbf A=(A,\leq)$ is the complete lattice $\big(\DM(\mathbf A),\subseteq\big)$ where
\[
\DM(\mathbf A):=\{L(B)\mid B\subseteq A\}=\{C\subseteq A\mid LU(C)=C\}.
\]
For Kleene posets, we can show the following:

\begin{example}
Consider the Kleene poset $\mathbf A$ depicted in Figure~7 such that 
$a'=d$ and $b'=c$:

\vspace*{2mm}

%\begin{center}
\setlength{\unitlength}{6mm}
\begin{tabular}{@{}c c c c c@{}}
\begin{picture}(5,10)
\put(3,1){\circle*{.3}}
\put(1,3){\circle*{.3}}
\put(5,3){\circle*{.3}}
\put(1,7){\circle*{.3}}
\put(5,7){\circle*{.3}}
\put(3,9){\circle*{.3}}
\put(3,1){\line(-1,1)2}
\put(3,1){\line(1,1)2}
\put(1,3){\line(0,1)4}
\put(1,3){\line(1,1)4}
\put(5,3){\line(-1,1)4}
\put(5,3){\line(0,1)4}
\put(3,9){\line(-1,-1)2}
\put(3,9){\line(1,-1)2}
\put(2.85,.25){$0$}
\put(.25,2.85){$a$}
\put(5.4,2.85){$b$}
\put(.25,6.85){$c$}
\put(5.4,6.85){$d$}
\put(2.85,9.4){$1$}
\put(2.2,-1.75){{\rm Fig.~7}}
\end{picture}&\phantom{xxxxx} &%
\begin{picture}(6,14)
	\put(3,1){\circle*{.3}}
	\put(1,3){\circle*{.3}}
	\put(5,3){\circle*{.3}}
	\put(1,7){\circle*{.3}}
	\put(5,7){\circle*{.3}}
	\put(1,11){\circle*{.3}}
	\put(5,11){\circle*{.3}}
	\put(3,13){\circle*{.3}}
	\put(3,1){\line(-1,1)2}
	\put(3,1){\line(1,1)2}
	\put(1,3){\line(0,1)8}
	\put(1,3){\line(1,1)4}
	\put(5,3){\line(-1,1)4}
	\put(5,3){\line(0,1)8}
	\put(1,7){\line(1,1)4}
	\put(5,7){\line(-1,1)4}
	\put(1,11){\line(1,1)2}
	\put(5,11){\line(-1,1)2}
	\put(2.35,.25){$(0,1)$}
	\put(-.8,2.85){$(0,c)$}
	\put(5.3,2.85){$(0,d)$}
	\put(-.8,6.85){$(a,b)$}
	\put(-.8,10.85){$(d,0)$}
	\put(5.3,6.85){$(b,a)$}
	\put(5.3,10.85){$(c,0)$}
	\put(2.35,13.4){$(1,0)$}
	\put(2.2,-1.75){{\rm Fig.~8}}
\end{picture}&\phantom{xxxxxxxxx} &%
\begin{picture}(6,14)
	\put(3,1){\circle*{.3}}
	\put(1,3){\circle*{.3}}
	\put(5,3){\circle*{.3}}
	\put(3,5){\circle*{.3}}
	\put(1,7){\circle*{.3}}
	\put(5,7){\circle*{.3}}
	\put(3,9){\circle*{.3}}
	\put(1,11){\circle*{.3}}
	\put(5,11){\circle*{.3}}
	\put(3,13){\circle*{.3}}
	\put(3,1){\line(-1,1)2}
	\put(3,1){\line(1,1)2}
	\put(1,3){\line(1,1)4}
	\put(5,3){\line(-1,1)4}
	\put(1,7){\line(1,1)4}
	\put(5,7){\line(-1,1)4}
	\put(1,11){\line(1,1)2}
	\put(5,11){\line(-1,1)2}
	\put(1.9,.25){$L\big((0,1)\big)$}
	\put(-1.9,2.85){$L\big((0,c)\big)$}
	\put(5.3,2.85){$L\big((0,d)\big)$}
	\put(3.3,4.85){$L\big((a,b),(b,a))\big)$}
	\put(-1.9,6.85){$L\big((a,b)\big)$}
	\put(-1.9,10.85){$L\big((d,0)\big)$}
	\put(5.3,6.85){$L\big((b,a)\big)$}
	\put(3.3,8.85){$L\big((d,0),(c,0)\big)$}
	\put(5.3,10.85){$L\big((c,0)\big)$}
	\put(1.9,13.4){$L\big((1,0)\big)$}
	\put(2.2,-1.75){{\rm Fig.~9}}
\end{picture}
\end{tabular}
%\end{center}

\vspace*{12mm}

Put $S:=\{a,b\}$. Then the poset $\big(P_S(\mathbf A),\sqsubseteq\big)$ is visualized in Figure~8.

The Dedekind-MacNeille completion $\BDM\big(P_S(\mathbf A),\sqsubseteq\big)$ of this poset is depicted in Figure~9, where
\begin{align*}
      L\big((0,1)\big) & =\{(0,1)\}, \\
      L\big((0,c)\big) & =\{(0,1),(0,c)\}, \\
      L\big((0,d)\big) & =\{(0,1),(0,d)\}, \\
L\big((a,b),(b,a)\big) & =\{(0,1),(0,c),(0,d)\}, \\
      L\big((a,b)\big) & =\{(0,1),(0,c),(0,d),(a,b)\}, \\
      L\big((b,a)\big) & =\{(0,1),(0,c),(0,d),(b,a)\}, \\
L\big((d,0),(c,0)\big) & =\{(0,1),(0,c),(0,d),(a,b),(b,a)\}, \\
      L\big((d,0)\big) & =\{(0,1),(0,c),(0,d),(a,b),(b,a),(d,0)\}, \\
      L\big((c,0)\big) & =\{(0,1),(0,c),(0,d),(a,b),(b,a),(c,0)\}, \\
      L\big((1,0)\big) & =\{(0,1),(0,c),(0,d),(a,b),(b,a),(d,0),(c,0),(1,0)\}.
\end{align*}
The Dedekind-MacNeille completion $\BDM(\mathbf A)$ of the given poset $\mathbf A$ is visualized in Figure~10; here the involution is given by 
$L(a)'=L(d)$, $L(b)'=L(c)$:

\vspace*{2mm}

\setlength{\unitlength}{6mm}
\begin{tabular}{@{}c c c c c@{}}
\phantom{xxx}&\begin{picture}(6,10)
\put(3,1){\circle*{.3}}
\put(1,3){\circle*{.3}}
\put(5,3){\circle*{.3}}
\put(3,5){\circle*{.3}}
\put(1,7){\circle*{.3}}
\put(5,7){\circle*{.3}}
\put(3,9){\circle*{.3}}
\put(3,1){\line(-1,1)2}
\put(3,1){\line(1,1)2}
\put(1,3){\line(1,1)4}
\put(5,3){\line(-1,1)4}
\put(3,9){\line(-1,-1)2}
\put(3,9){\line(1,-1)2}
\put(2.4,.25){$L(0)$}
\put(-.7,2.85){$L(a)$}
\put(5.3,2.85){$L(b)$}
\put(3.3,4.85){$L(c,d)$}
\put(-.7,6.85){$L(c)$}
\put(5.3,6.85){$L(d)$}
\put(2.4,9.4){$L(1)$}
\put(2.2,-1.75){{\rm Fig.~10}}
\end{picture}&\phantom{xxxxxxxxxxxxxxxxxxxx} &\begin{picture}(6,14)
\put(3,1){\circle*{.3}}
\put(1,3){\circle*{.3}}
\put(5,3){\circle*{.3}}
\put(3,5){\circle*{.3}}
\put(1,7){\circle*{.3}}
\put(5,7){\circle*{.3}}
\put(3,9){\circle*{.3}}
\put(1,11){\circle*{.3}}
\put(5,11){\circle*{.3}}
\put(3,13){\circle*{.3}}
\put(3,1){\line(-1,1)2}
\put(3,1){\line(1,1)2}
\put(1,3){\line(1,1)4}
\put(5,3){\line(-1,1)4}
\put(1,7){\line(1,1)4}
\put(5,7){\line(-1,1)4}
\put(1,11){\line(1,1)2}
\put(5,11){\line(-1,1)2}
\put(1.5,.25){$\big(L(0),L(1)\big)$}
\put(-3.05,2.85){$\big(L(0),L(c)\big)$}
\put(5.3,2.85){$\big(L(0),L(d)\big)$}
\put(3.3,4.85){$\big(L(0),L(c,d)\big)$}
\put(-3.05,6.85){$\big(L(a),L(b)\big)$}
\put(-3.05,10.85){$\big(L(d),L(0)\big)$}
\put(5.3,6.85){$\big(L(b),L(a)\big)$}
\put(3.3,8.85){$\big(L(c,d),L(0)\big)$}
\put(5.3,10.85){$\big(L(c),L(0)\big)$}
\put(1.5,13.4){$\big(L(1),L(0)\big)$}
\put(2.2,-1.75){{\rm Fig.~11}}
\end{picture} &
\end{tabular}

\vspace*{12mm}

Finally, the lattice $\Big(P_{\{L(s)\mid s\in S\}}\big(\BDM(\mathbf A)\big),\sqcup,\sqcap\Big)$ is depicted in Figure~11, where

\begin{align*}
  \big(L(0),L(1)\big) & =(\{0\},\{0,a,b,c,d,1\}), \\
  \big(L(0),L(c)\big) & =(\{0\},\{0,a,b,c\}), \\
  \big(L(0),L(d)\big) & =(\{0\},\{0,a,b,d\}), \\
\big(L(0),L(c,d)\big) & =(\{0\},\{0,a,b\}), \\
  \big(L(a),L(b)\big) & =(\{0,a\},\{0,b\}), \\
  \big(L(b),L(a)\big) & =(\{0,b\},\{0,a\}), \\
\big(L(c,d),L(0)\big) & =(\{0,a,b\},\{0\}), \\
  \big(L(d),L(0)\big) & =(\{0,a,b,d\},\{0\}), \\
  \big(L(c),L(0)\big) & =(\{0,a,b,c\},\{0\}), \\
  \big(L(1),L(0)\big) & =(\{0,a,b,c,d,1\},\{0\}).
\end{align*}
Hence, in this case, the lattices $\Big(P_{\{L(s)\mid s\in S\}}\big(\BDM(\mathbf A)\big),\sqcup,\sqcap\Big)$ and $\BDM\big(P_S(\mathbf A),\sqsubseteq\big)$ are isomorphic.
\end{example}

The lattices mentioned above need not be isomorphic for distributive posets 
$\mathbf A$, which are not Kleene posets.

%For distributive posets $\mathbf A$, which are not Kleene posets, the 
%lattices mentioned above  need not be isomorphic.

\begin{example}\label{counterdm}
	Consider the distributive poset $\mathbf A$ which is not a Kleene poset depicted in Figure~12:
	
	\vspace*{8mm}
	
		\setlength{\unitlength}{10mm}
		\begin{tabular}{@{}c c c c c c@{}}
			\phantom{xxx}
		&\begin{picture}(2,4)
	%	\put(1,0){\circle*{.3}}
	%	\put(1,1){\circle*{.3}}
		\put(0,2){\circle*{.3}}
		\put(2,2){\circle*{.3}}
		\put(1,3){\circle*{.3}}
		\put(1,4){\circle*{.3}}
	%	\put(1,1){\line(-1,1)1}
	%	\put(1,1){\line(0,-1)1}
	%	\put(1,1){\line(1,1)1}
		\put(1,3){\line(-1,-1)1}
		\put(1,3){\line(1,-1)1}
		\put(1,3){\line(0,1)1}
	%	\put(.85,-.6){$0$}
	%	\put(1.35,.8){$a$}
		\put(-.6,1.8){$c$}
		\put(2.35,1.8){$d$}
		\put(1.35,2.8){$b$}
		\put(.85,4.3){$1$}
		\put(.4,1.0){{\rm Fig.~12}}
	\end{picture}&\phantom{xxxxxxxxxx}&	\begin{picture}(2,3)
	%	\put(1,0){\circle*{.3}}
	%	\put(1,1){\circle*{.3}}
	\put(0,3){\circle*{.3}}
	\put(2,3){\circle*{.3}}
	%	\put(1,3){\circle*{.3}}
	%	\put(1,4){\circle*{.3}}
	%	\put(1,1){\line(-1,1)1}
	%	\put(1,1){\line(0,-1)1}
	%	\put(1,1){\line(1,1)1}
	%	\put(1,3){\line(-1,-1)1}
	%	\put(1,3){\line(1,-1)1}
	%	\put(1,3){\line(0,1)1}
	%	\put(.85,-.6){$0$}
	%	\put(1.35,.8){$a$}
	\put(-1.2096,2.8){$(c,d)$}
	\put(2.35,2.8){$(d,c)$}
	%		\put(1.35,2.8){$b$}
	%		\put(.85,4.3){$1$}
	\put(.4,1.0){{\rm Fig.~13}}
\end{picture}
&\phantom{xxxxxxxxxxxxxxxxx} &\begin{picture}(2,4)
	%	\put(1,0){\circle*{.3}}
	\put(1,2){\circle*{.3}}
	\put(0,3){\circle*{.3}}
	\put(2,3){\circle*{.3}}
	\put(1,4){\circle*{.3}}
	%	\put(1,4){\circle*{.3}}
	\put(1,2){\line(-1,1)1}
	%	\put(1,1){\line(0,-1)1}
	\put(1,2){\line(1,1)1}
	\put(1,4){\line(-1,-1)1}
	\put(1,4){\line(1,-1)1}
	\put(1.35,1.8){$\emptyset=LU(\emptyset)$}
	\put(-1.45096,2.8){$L(c,d)$}
	\put(2.35,2.8){$L(d,c)$}
	\put(-0.385,4.3){$LU\big((c,d), (d,c)\big)$}
	%		\put(.85,4.3){$1$}
	\put(.4,1.0){{\rm Fig.~14}}
\end{picture}
	\end{tabular}

\vspace*{-2mm}

Put $S:=\{c,d\}$. Then the poset $\big(P_S(\mathbf A),\sqsubseteq\big)$ is visualized in Figure~13.

The Dedekind-MacNeille completion $\BDM\big(P_S(\mathbf A),\sqsubseteq\big)$ of this poset is depicted in Figure~14 and the Dedekind-MacNeille completion $\BDM(\mathbf A)$ of the given poset $\mathbf A$ is visualized in Figure~15:

\vspace*{12mm}

	\setlength{\unitlength}{8mm}
	\begin{tabular}{@{}c c c c c@{}}
	\phantom{xxxxxxxxxx}&	\begin{picture}(2,4)
		%	\put(1,0){\circle*{.3}}
			\put(1,1){\circle*{.3}}
			\put(0,2){\circle*{.3}}
			\put(2,2){\circle*{.3}}
			\put(1,3){\circle*{.3}}
			\put(1,4){\circle*{.3}}
			\put(1,1){\line(-1,1)1}
		%	\put(1,1){\line(0,-1)1}
			\put(1,1){\line(1,1)1}
			\put(1,3){\line(-1,-1)1}
			\put(1,3){\line(1,-1)1}
			\put(1,3){\line(0,1)1}
		%	\put(.85,-.6){$0$}
			\put(1.35,.8){$\emptyset=L(c,d)$}
			\put(-1.098,1.8){$L(c)$}
			\put(2.35,1.8){$L(d)$}
			\put(1.35,2.8){$L(b)$}
			\put(.85,4.3){$L(1)$}
			\put(.4,-1.4){{\rm Fig.~15}}
		\end{picture}&\phantom{xxxxxxxxxxxxxxxxxxxxxxxxxxx}&%
	\begin{picture}(2,4)
		\put(1,0){\circle*{.3}}
		\put(1,1){\circle*{.3}}
		\put(0,2){\circle*{.3}}
		\put(2,2){\circle*{.3}}
		\put(1,3){\circle*{.3}}
		\put(1,4){\circle*{.3}}
		\put(1,1){\line(-1,1)1}
		\put(1,1){\line(0,-1)1}
		\put(1,1){\line(1,1)1}
		\put(1,3){\line(-1,-1)1}
		\put(1,3){\line(1,-1)1}
		\put(1,3){\line(0,1)1}
		\put(1.5,-.6){$(\emptyset, L(1))$}
		\put(1.5,.8){$(\emptyset, L(b))$}
		\put(-2.9296,1.8){$(L(c), L(d))$}
		\put(2.35,1.8){$(L(d), L(c))$}
		\put(1.5,2.8){$(L(b), \emptyset)$}
		\put(1.5,4.3){$(L(1), \emptyset)$}
		\put(.2,-1.4){{\rm Fig.~16}}
	\end{picture}&
\end{tabular}
\vspace*{12mm}

Finally, the lattice $\Big(P_{\{L(s)\mid s\in S\}}\big(\BDM(\mathbf A)\big),\sqcup,\sqcap\Big)$ is depicted in Figure~16.

Hence, in this case, the lattices $\Big(P_{\{L(s)\mid s\in S\}}\big(\BDM(\mathbf A)\big),\sqcup,\sqcap\Big)$ and $\BDM\big(P_S(\mathbf A),\sqsubseteq\big)$ 
are not isomorphic.
\end{example}

Let  $\mathbf A=(A,\leq)$ be a distributive poset and let $Fin(A)$ denote the set of all finite subsets of $A$. We put (see \cite{Niederle}) 
$$
G(\mathbf A) := \{L\big(U(A_1),...,U(A_n)\big)\mid  n \in \mathbb N_{+} \& 
\forall i, 1\leq i\leq n, \emptyset\not= A_i \in Fin(\mathbf A)\}. 
$$

Note that $G(\mathbf  A)$ is a subset of $\BDM(\mathbf A)$,  
containing all principal ideals. It is worth noticing that $G(\mathbf A)=\BDM(\mathbf A)$ provided $A$ is finite. Moreover, from (\cite[Proposition~31]{Niederle}) we immediately 
obtain that any element of $G(\mathbf A)$ is also of the form $LU\big(L(B_1),...,L(B_n)\big)$ where 
$B_i$ are finite non-empty subsets of $A$.

The following definition and theorem are motivated by a similar result of Niederle for Boolean posets (\cite[Theorem~17]{Niederle}).

\begin{definition}\label{doubly}
	\begin{enumerate}
		\item Let $\mathbf A=(A,\leq)$ be a  poset.  A subset $X$ of $A$ is called {\em  doubly dense in $\mathbf A$} if $a=\bigvee_{\mathbf A}\big(L(a)\cap X\big)=\bigwedge_{\mathbf A}\big(U(a)\cap X\big)$ for all $a\in A$.

		\item Let $\mathbf A=(A,\leq,{}')$ be a poset with an antitone involution $'$. A subset $X$ of $A$ is called {\em involution-closed and doubly dense in $\mathbf A$} if $X'\subseteq X$ and $X$ is doubly dense in $\mathbf A$.
	\end{enumerate}
\end{definition}

We will need the following

\begin{proposition}[\cite{Niederle}, Proposition 33]\label{niedga}
	Let  $\mathbf A=(A,\leq)$ be a distributive poset. Then 
	$\big(G(\mathbf A),\subseteq\big)$  
	is a distributive lattice and $X= \{L(a) \mid a \in A\}$ is doubly dense in 
	$G(\mathbf A)$, 
	generates $G(\mathbf A)$ and $(X, \subseteq)$ is isomorphic to $\mathbf A$. 
\end{proposition}

In what follows, if $\mathbf A=(A,\leq,{}')$ is a  
poset with an antitone involution $'$ and $X\subseteq A$, we define:
\begin{itemize}
	\item $X':=\{x'\in A\mid x\in X\}$,
	\item $X^\bot:=\{a\in A\mid a\leq x'\text{ for all }x\in X\}=L(X')$.
\end{itemize}

\begin{remark}\label{ddcd}\em 
	Recall that any involution-closed and doubly dense subset $X$ in $\mathbf A$ is a poset with induced order and involution. Moreover, if $\mathbf A=(A,\leq,{}')$ is a poset with an antitone involution $'$ then $A$ is an involution-closed and doubly dense subset in its Dedekind-MacNeille completion $\BDM(\mathbf A)$ 
	with involution ${}^\bot$. This can be shown by the same arguments as in {\rm(\cite[Theorem 16]{Niederle})}, so we omit it.
\end{remark}

By the preceding remark,  Proposition \ref{niedga}  and 
\cite[Theorem 34]{Niederle}  we have the following

\begin{corollary}\label{dstinv}
	Let  $\mathbf A=(A,\leq,\, {}')$ be a distributive poset  with an antitone involution $'$. 
	Then $\big(G(\mathbf A),\subseteq, {}^{\bot}\big)$  
	is a distributive lattice  with an antitone involution 
	${}^{\bot}$ and $X= \{L(a) \mid a \in A\}$ is  involution-closed  and doubly dense in 
	$G(\mathbf A)$, 
	generates $G(\mathbf A)$ and $(X, \subseteq, {}^{\bot})$ is isomorphic to $\mathbf A$. 
\end{corollary}

\begin{corollary}\label{embeddstinv}
	{\bfseries Embedding theorem for distributive posets with an antitone involution.} The following 
	conditions are equivalent for a poset $\mathbf A$: 
	\begin{enumerate}[{\rm(i)}]
		\item $\mathbf A$ is a distributive poset with an antitone involution; 
		\item  $\mathbf A$ is an involution-closed and doubly dense subset of 
		a distributive lattice with an antitone involution.
	\end{enumerate}
\end{corollary}

But we can prove more. 

\begin{proposition}\label{dstkleene}
	Let  $\mathbf A=(A,\leq,\, {}')$ be a Kleene poset. 
	Then $(G\big(\mathbf A),\subseteq, {}^{\bot}\big)$  
	is a  Kleene lattice and $X= \{L(a) \mid a \in A\}$ is  involution-closed  and doubly dense in 
	$G(\mathbf A)$, 
	generates $G(\mathbf A)$ and $(X, \subseteq, {}^{\bot})$ is isomorphic to $\mathbf A$. 
\end{proposition}
\begin{proof} It is enough to check that for all $C, D\in G(\mathbf A)$, we have 
	$$
	C\cap {C}^{\bot} \subseteq L\big(U(D\cup D^{\bot})\big). 
	$$
	Assume first that $C=LU(E)$ and $D=LU(F)$ where $E, F$ are non-empty finite subsets of $A$.
	Then $C=\bigvee \{L(e)\mid e\in E\}$ and $D=\bigvee \{L(f)\mid f\in F\}$. We compute:
	
	\begin{align*}
		&\big(\bigvee \{L(e)\mid e\in E\}\big)\wedge %
		\big(\bigwedge \{L(g)^{\bot}\mid g\in E\}\big)=
		\bigvee_{e\in E} \Big(L(e)\wedge  %
		\big(\bigwedge \{L(g)^{\bot}\mid g\in E\}\big)\Big)\\
		&\leq \bigvee_{e\in E} \big(L(e)\wedge L(e)^{\bot}\big)=%
		\bigvee_{e\in E} \big(L(e)\wedge L(e')\big)%
		=\bigvee_{e\in E} L(e,e')\leq \bigwedge_{h\in F} LU(h,h')\\
		&=\bigwedge_{h\in F} \big(L(h)\vee L(h')\big)\leq\bigwedge_{h\in F} %	
		\Big(\big(\bigvee \{L(f)\mid f\in F\}\big)\vee L(h')\Big)\\[0.2cm]
		&=\big(\bigvee \{L(f)\mid f\in F\}\big)\wedge %
		\big(\bigwedge \{L(h)^{\bot}\mid h\in F\}\big).
	\end{align*}
	
	Now, assume that $C=\bigwedge_{i=1}^{n} C_i$, $D=\bigwedge_{j=1}^{m} D_i$ where 
	$C_i=LU(E_i)$, $D_j=LU(F_j)$, $E_i$ and  $F_j$ are non-empty finite subsets 
	of $A$, $1\leq i\leq n$ 
	and $1\leq j\leq m$. We compute:
	$$
	\begin{array}{@{}r c l}
		(\bigwedge_{i=1}^{n} C_i)\wedge 	(\bigvee_{k=1}^{n} C_k^{\bot})&=&
		\bigvee_{k=1}^{n} \big(C_k^{\bot} \wedge (\bigwedge_{i=1}^{n} C_i)\big)\leq % 
		\bigvee_{k=1}^{n} (C_k^{\bot} \wedge C_k)\leq %
		\bigwedge_{l=1}^{m} (D_l^{\bot} \vee D_l)\\[0.2cm]
&\leq&\bigwedge_{l=1}^{m} \big((\bigvee_{j=1}^{m}  D_j^{\bot}) \vee D_l\big)=%
		(\bigvee_{j=1}^{m}  D_j^{\bot})\vee (\bigwedge_{l=1}^{m} D_l).
	\end{array}
	$$
\end{proof}

\begin{theorem}\label{embeddstkleene}
	{\bfseries Embedding theorem for Kleene posets.} The following 
	conditions are equi\-va\-lent for a poset $\mathbf A$: 
	\begin{enumerate}[{\rm(i)}]
		\item $\mathbf A$ is a Kleene poset; 
		\item  $\mathbf A$ is an involution-closed and doubly dense subset of 
		a Kleene lattice.
	\end{enumerate}
\end{theorem}
\begin{proof} (i)$\implik$ (ii) has been proved in Proposition \ref{dstkleene}. 
	
	(ii)$\implik$ (i): From Corollary \ref{embeddstinv}, we know that $\mathbf A$ is a distributive 
	poset with an antitone involution ${}'$. But the involution reflects the Kleene condition.	
	Namely, let $x, y\in A$. Assume that $a\in L(x,x')$ and $b\in U(y,y')$. Then 
	$a\leq x\wedge x'\leq y\vee y'\leq b$ in the Kleene lattice. Hence $a\leq b$  in $\mathbf A$, i.e., 
	 $\mathbf A$ is a Kleene poset. 
\end{proof}

\section*{Data availability statement}

Data sharing is not applicable to this article as no datasets were generated or analysed during the current study.

Authors' addresses:

Ivan Chajda \\
Palack\'y University Olomouc \\
Faculty of Science \\
Department of Algebra and Geometry \\
17.\ listopadu 12 \\
771 46 Olomouc \\
Czech Republic \\
ivan.chajda@upol.cz

Helmut L\"anger \\
TU Wien \\
Faculty of Mathematics and Geoinformation \\
Institute of Discrete Mathematics and Geometry \\
Wiedner Hauptstra\ss e 8-10 \\
1040 Vienna \\
Austria, and \\
Palack\'y University Olomouc \\
Faculty of Science \\
Department of Algebra and Geometry \\
17.\ listopadu 12 \\
771 46 Olomouc \\
Czech Republic \\
helmut.laenger@tuwien.ac.at

Jan Paseka \\
Masaryk University Brno \\
Faculty of Science \\
Department of Mathematics and Statistics \\
Kotl\'a\v rsk\'a 2 \\
611 37 Brno \\
Czech Republic \\
paseka@math.muni.cz
\end{document}